\documentclass[11pt,a4paper]{article}
\usepackage{indentfirst}
\usepackage{amsfonts}
\usepackage{amssymb}
\usepackage{mathrsfs}
\usepackage{amsmath}
\usepackage{amsthm}
\usepackage{enumerate}
\usepackage{cite}
\usepackage{mathrsfs}
\allowdisplaybreaks
\usepackage{geometry}
\usepackage{ifpdf}
\ifpdf
  \usepackage[colorlinks=true,linkcolor=blue,citecolor=red, final,backref=page,hyperindex]{hyperref}
\else
  \usepackage[colorlinks,final,backref=page,hyperindex,hypertex]{hyperref}
\fi

\usepackage{txfonts}
\usepackage{indentfirst}
\usepackage{latexsym}

\def\<{\langle}
\def\>{\rangle}

\def\c{\cdot}

\newtheorem{thm}{Theorem}[section]
\newtheorem{lem}[thm]{Lemma}
\newtheorem{cor}[thm]{Corollary}
\newtheorem{pro}[thm]{Proposition}
\newtheorem{ex}[thm]{Example}

\theoremstyle{definition}
\newtheorem{defi}{Definition}[section]

\theoremstyle{remark}
\newtheorem{rmk}{Remark}[section]
\begin{document}
\title{\bf On $n$-pre-Lie algebras and dendrification of $n$-Lie algebras}
\author{\bf T. Chtioui, A. Hajjaji, S. Mabrouk, A. Makhlouf}
\author{{ Taoufik Chtioui$^{1}$
 \footnote {  E-mail: chtioui.taoufik@yahoo.fr}
,\  Atef Hajjaji$^{1}$
    \footnote {E-mail:  atefhajjaji100@gmail.com}
,\  Sami Mabrouk$^{2}$
 \footnote {   E-mail: mabrouksami00@yahoo.fr}
\ and Abdenacer Makhlouf$^{3}$
 \footnote { Corresponding author,  E-mail: Abdenacer.Makhlouf@uha.fr}
}\\
{\small 1.  University of Sfax, Faculty of Sciences Sfax,  BP
1171, 3038 Sfax, Tunisia} \\
{\small 2.  University of Gafsa, Faculty of Sciences Gafsa, 2112 Gafsa, Tunisia}\\
{\small 3.~ IRIMAS - Département de Mathématiques, 6, rue des frères Lumière,
F-68093 Mulhouse, France}}
\date{}
\maketitle

\begin{abstract}
The main purpose of this paper  is to introduce the notion of $n$-L-dendriform algebra which can be seen as a dendrification of $n$-pre-Lie algebras by means of  $\mathcal{O}$-operators. We investigate the representation theory of $n$-pre-Lie algebras and provide some related constructions. Furthermore, we introduce the notion of  phase space of a $n$-Lie algebra and show
that a $n$-Lie algebra has a phase space if and only if it is sub-adjacent to a $n$-pre-Lie algebra. Moreover, we present a procedure to construct $(n + 1)$-pre-Lie
algebras from $n$-pre-Lie algebras equipped with a generalized trace function. 
\end{abstract}



\tableofcontents

\numberwithin{equation}{section}
\section{Introduction}
The notion
of an $n$-Lie algebra, or a Filippov algebra was introduced in \cite{Filippov}. It is the algebraic structure corresponding
to Nambu mechanics \cite{Nambu,Takhtajan}. Notice that $n$-Lie algebras, or more generally, $n$-Leibniz algebras, have attracted
attention from several fields of mathematics and physics due to their close connection with dynamics,
geometries as well as string theory \cite{Bagger1,BaggerLambert}, see also \cite{Casas,Cherkis,Figueroa1,Ho,Papadopoulos}. For example, the structure of
$n$-Lie algebra is applied to  study  supersymmetry and gauge symmetry transformations of the
word-volume theory of multiple $M2$-branes and the generalized identity for $n$-Lie algebras can be regarded
as a generalized Plucker relation in the physics literature. See the review article \cite{Azcarraga} for more details.

A dendriform algebra is a module equipped with two binary products whose sum is associative. This
concept was introduced by Loday in the late 1990s in the study of periodicity in algebraic K-theory \cite{L1}.
Several years later, Loday and Ronco introduced the concept of  tridendriform algebra from their study of
algebraic topology \cite{AL1}. It is a module with three binary operations whose sum is associative. Afterward,
quite  few similar algebraic structures were introduced, such as the quadri-algebra and ennea algebra \cite{AL}.
The notion of splitting of associativity was introduced by Loday to describe this phenomena in general, for
the associative operation (see \cite{L6,L5,L4}, for more details).

Pre-Lie algebras are a class of nonassociative algebras coming from the study of convex homogeneous
cones, affine manifolds and affine structures on Lie groups, and the cohomologies of associative algebras.
See the survey \cite{burde} and the references therein for more details. The main properties  of a pre-Lie algebra is that
the commutator gives rise to a Lie algebra and the left multiplication gives rise to a representation of the
commutator Lie algebra. Conversely, a relative Rota-Baxter operator action on a Lie algebra gives rise to a
pre-Lie algebra \cite{Golubchik,Kupershmidt} and thus the pre-Lie algebra can be seen as the underlying algebraic structure of
a relative Rota-Baxter operator. In \cite{Ming}, the authors introduce the notion of  $n$-pre-Lie algebra, which gives
a $n$-Lie algebra naturally and its left multiplication operator gives rise to a representation of this $n$-Lie
algebra. A $n$-pre-Lie algebra can also be obtained through the action of a relative Rota-Baxter operator
on a $n$-Lie algebra. For ($n=3$), see \cite{Bai1} for more details. 

Note that the notion of a symplectic structure on a  $n$-Lie algebra was introduced in \cite{Ming} and
it is shown that the underlying structure of a symplectic $n$-Lie algebra is a quadratic $n$-pre-Lie
algebra. In this paper, we introduce the notion of a phase space of a $n$-Lie algebra $A$, which is a symplectic $n$-Lie
algebra $A\oplus A^*$ satisfying some conditions, and show that a $n$-Lie algebra has a phase space if and
only if it is sub-adjacent to a $n$-pre-Lie algebra. We also introduce the notion of  Manin triple
of $n$-pre-Lie algebras and show that there is a one-to-one correspondence between Manin triples of
$n$-pre-Lie algebras and phase spaces of $n$-Lie algebras.

In \cite{Awata}, the authors introduce a realization of the quantum Nambu bracket in terms
of matrices (using the commutator and the trace of matrices). This construction was
generalized in \cite{Makhlouf1} to the case of any Lie algebra where the commutator is replaced
by the Lie bracket, and the matrix trace is replaced by linear forms having similar
properties, that led to a so  called $3$-Lie algebras induced by Lie algebras. This construction
is generalized to the case of $n$-Lie algebras in \cite{Makhlouf2},  see also \cite{Makhlouf3}.

The following is an outline of the paper. In Section 2, we summarize some definitions and known results
about $n$-Lie algebras and $n$-pre-Lie algebras which will be useful in the sequel. Section 3 is dedicated to  study representations of $n$-pre-Lie algebras. In Section 4,
we introduce the notion of  phase space of a $n$-Lie algebra and show that a $n$-Lie algebra has a
phase space if and only if it is sub-adjacent to a $n$-pre-Lie algebra. We also introduce the notion of
 Manin triple of $n$-pre-Lie algebras and study its relationship with phase spaces of $n$-Lie algebras. Next, we provide a construction
procedure of $(n + 1)$-pre-Lie algebras starting from an $n$-pre-Lie algebra
and a trace map. In Section 5, we introduce the
notion of $n$-L-dendriform algebra. We establish that it has two associated $n$-pre-Lie algebras (horizontal and
vertical). They have the same sub-adjacent $n$-Lie algebra. In addition, the left multiplication operator of the
first operation (called north-west: $\nwarrow$) and the right multiplication operator of the second operation (called
north-east:$\nearrow$ ) lead to  a bimodule of the associated horizontal $n$-pre-Lie algebra.

Throughout this paper $\mathbb{K}$ is a field of characteristic $0$ and all vector spaces are over $\mathbb{K}$ and finite dimensional.
\section{Preliminaries and Basics}
\label{sec:bas}

In this section, we give some preliminaries and basic results on $n$-Lie algebras and $n$-pre-Lie algebras.
\begin{defi}\cite{Filippov}
  A $n$-Lie algebra  consists  of a vector space $A$ equipped with a skew-symmetric $n$-linear map $[\cdot,\cdots,\cdot]:
 A^n\rightarrow A$ such that the following Filippov-Jocobi identity holds (for $x_i,y_i\in A, 1\leq i\leq {n}$)
\begin{align}\label{eq:de1}
&[x_1,\cdots,x_{n-1},[y_1,\cdots,y_{n}]]
=\sum_{i=1}^{n}[y_1,\cdots,y_{i-1},[x_1,\cdots,x_{n-1},y_i],y_{i+1},\cdots,y_{n}].
\end{align}
\end{defi}
Recall that the $n$-linear bracket is said to be skew-symmetric if for any permutation  $\sigma$ in the permutation group $S_n$ and $x_1, \cdots , x_n \in A$, we have $$[x_{\sigma(1)},\cdots, x_{\sigma(n)}
] = Sgn(\sigma)[x_1, \cdots, x_n],$$
where $Sgn(\sigma)$ is the signature of $\sigma$.

The Filippov-Jacobi identity can expressed as follows. For $X=(x_1,\cdots,x_{n-1})\in A^{n-1}$, the operator
\begin{align}\label{eq:adjoint}
ad(X):A\to A, \quad ad(X)(y):=[x_1,\cdots,x_{n-1},y], \quad \forall y\in A,
\end{align}
is a derivation in the sense that
\begin{equation}
ad(X)([y_1,\cdots,y_{n}])=\sum_{i=1}^{n}[y_1,\cdots,y_{i-1},ad(X)(y_i),y_{i+1},\cdots,y_{n}].
\end{equation}
A morphism between $n$-Lie algebras is a linear map that preserves the $n$-Lie brackets.
\begin{ex}\cite{Filippov} Consider $(n+1)$-dimensional $n$-Lie algebra $A$ generated by $(e_1,\cdots,e_{n+1})$  with  the following multiplication
\begin{equation*}
[e_1,\cdots,\widehat{e}_i,\cdots,e_n]= e_i,
\end{equation*}
where $\widehat{e}_i$ means that the element $e_i$ is omitted.
\end{ex}

The notion of a representation of an $n$-Lie algebra was introduced in \cite{repKasymov},  see also \cite{rep}.
\begin{defi}\label{defi:rep}
 A representation of a $n$-Lie algebra $(A,[\cdot,\cdots,\cdot])$ on a vector space $V$ is a skew-symmetric linear map $\rho: \wedge^{n-1}A\rightarrow gl(V)$ satisfying
\begin{align}\label{RepNLie1}
&\rho ([x_1,\cdots,x_n],y_1,\cdots,y_{n-2})=
\sum_{i=1}^{n}(-1)^{n-i}\rho(x_1,\cdots,\widehat{x}_i,\cdots,x_n)\rho(x_i,y_1,\cdots,y_{n-2}),\\&[\rho (x_1,\cdots,x_{n-1}),\rho(y_1,\cdots,y_{n-1})]
=\label{RepNLie2}\sum_{i=1}^{n-1}\rho(y_1,\cdots,y_{i-1},[x_1,\cdots,x_{n-1},y_i],y_{i+1},\cdots,y_{n-1}),\end{align}

for $x_i,y_i\in A, 1\leq i\leq n$.\\ We refer to the representation  by the pair  $(V,\rho)$. We say also that  $V$ is an $A$-module.
Let $\rho(x_1, \cdots , x_{n-1})= ad (x_1,\cdots , x_{n-1})$ for $x_1, \cdots, x_{n-1}\in A$. Then $(A, ad)$ is an $A$-module
and is called the adjoint module of $A$.
\end{defi}
\begin{lem}\label{lem:sepro}
$(V,\rho)$ is a representation of a $n$-Lie algebra $(A,[\cdot,\cdots,\cdot])$ if and only if there is a $n$-Lie algebra structure
on the direct sum $A\oplus V$ of the underlying vector spaces $A$ and $V$ given by
\begin{equation}\label{eq:sum}
[x_1+u_1,\cdots,x_n+u_n]_{A\oplus V}=[x_1,\cdots,x_n]+\sum_{i=1}^{n}(-1)^{n-i}\rho(x_1,\cdots,\widehat{x}_i,\cdots,x_n)(u_i),
\end{equation}
for $x_i\in A, u_i\in V, 1\leq i\leq n$, and $\widehat{x}_i$ means that $x_i$ is omitted. We denote it by $A\ltimes_\rho V.$
\end{lem}
\ \ \ Let $(V,\rho)$ be a representation of a $n$-Lie algebra $A$ and $V^{*}$ be the dual space of $V$. Define $\rho^{*}:\wedge^{n-1}A\rightarrow gl(V^{*})$ by
\begin{equation}\label{eq:dual}
\langle\rho^{*}(x_1,\cdots,x_{n-1})\xi,v\rangle=-\langle\xi,\rho(x_1,\cdots,x_{n-1})v\rangle,
\end{equation}
$\forall \xi\in V^{*},\;x_i\in A,1\leq i\leq{n-1},\;v\in V.$
\begin{lem}\label{lem:rep}
Let $(V,\rho)$ be a representation of a $n$-Lie algebra $(A,[\cdot,\cdots,\cdot])$. Then $(V^{*},\rho^{*})$ is a representation of $(A,[\cdot,\cdots,\cdot])$, which is called the dual
representation.
\end{lem}
\begin{ex}
Let $A$ be a $n$-Lie algebra. Let $(A, ad)$ be the adjoint representation
of $A$ given by Eq.~\eqref{eq:adjoint}. The dual representation of the adjoint
representation $(A, ad)$ is denoted by $(A^{*}, ad^{*})$, where $A^{*}$ is the dual space of $A$, and $ad^{*} : \wedge^{n-1}A\rightarrow End(A^{*})$, for all
 $x_1,\cdots, x_{n-1}\in A$, $x\in A$ and $\xi\in A^{*}$,
\begin{equation*}
\langle ad^{*}(x_1,\cdots,x_{n-1})\xi,x\rangle=-\langle \xi,ad(x_1,\cdots,x_{n-1})x\rangle,
\end{equation*}
 and is called the coadjoint representation.
For any $n$-Lie algebra $A$, we have the semi-direct product $n$-Lie algebra $(A\oplus A^{*},[\cdot,\cdots,\cdot]^{*})$, where
$[\cdot,\cdots,\cdot]^{*}:(A\oplus A^{*})^{ n}\rightarrow A\oplus A^{*}$, for all $x_i\in A$, $\xi_i\in A^{*}$, $1\leq i\leq n$,
\begin{equation*}
[x_1+\xi_1,\cdots,x_n+\xi_n]^{*}=[x_1,\cdots,x_n]+\sum_{i=1}^{n}(-1)^{n-i}ad^{*}(x_1,\cdots,\widehat{x}_i,\cdots,x_n)(\xi_i),
\end{equation*}
where $\widehat{x}_i$ means that $x_i$ is omitted.
\end{ex}

Now we recall the definition of $n$-pre-Lie algebra and exhibit construction results in terms of $\mathcal{O}$-operators
on $n$-Lie algebras (see \cite{Ming} for more details)

\begin{defi}\cite{Liu}\label{defi:o}
Let $(A,[\cdot,\cdots,\cdot])$ be a $n$-Lie algebra and $(V,\rho)$
a representation.  A linear operator $T:V\rightarrow A$ is called
an $\mathcal O$-operator associated to $( V,\rho)$ if $T$
satisfies
\begin{equation}\label{eq:Ooperator}
 [Tu_1,\cdots,Tu_n]=T\Big(\sum_{i=1}^{n}(-1)^{n-i}\rho(Tu_1,\cdots,\widehat{Tu_i},\cdots,Tu_n)u_i\Big),\quad \forall u_i\in V,1\leq i\leq n.
\end{equation}
An $\mathcal{O}$-operator associated to
the adjoint representation $(A,ad)$ is called a Rota-Baxter operator of weight $\lambda=0$.
\end{defi}
\begin{defi}
 Let $A$ be a vector space with a multilinear map $\{\c,\cdots,\c\}:(\wedge^{n-1}A)\otimes A \rightarrow A$. The pair
$(A, \{\c, \cdots ,\c\})$ is called a $n$-pre-Lie algebra if for all  $x_1,\cdots , x_{n},y_1,\cdots,y_n\in A$, the following identities hold:
{\small\begin{eqnarray}
\{x_1,\cdots,x_{n-1},\{y_1,\cdots,y_n\}\}&=&\sum_{i=1}^{n-1}\{y_1,\cdots,y_{i-1},[x_1,\cdots,x_{n-1},y_i]^C,y_{i+1},\cdots,y_n\}\nonumber\\
&&+\{y_1,\cdots,y_{n-1},\{x_1,\cdots,x_{n-1},y_n\}\},\label{n-pre-Lie 1}\\
\{ [x_1,\cdots,x_n]^C,y_1,\cdots, y_{n-1}\}&=&\sum_{i=1}^{n}(-1)^{n-i}\{x_1,\cdots,\widehat{x}_i,\cdots,x_{n},\{ x_i,y_1,\cdots,y_{n-1}\}\},\label{n-pre-Lie 2}
\end{eqnarray}}
 where  $[\cdot,\cdots,\cdot]^C$ is defined by
\begin{equation}
[x_1,\cdots,x_n]^C=\sum_{i=1}^{n}(-1)^{n-i}\{x_1,\cdots,\widehat{x_i},\cdots,x_n,x_i\},\quad \forall  x_i\in A,1\leq i\leq n.\label{eq:ncc}
\end{equation}
\end{defi}
\begin{pro}\label{pro:comm}
Let $(A,\{\cdot,\cdots,\cdot\})$ be a $n$-pre-Lie algebra. Then the induced $n$-commutator given by Eq.~\eqref{eq:ncc} defines
a $n$-Lie algebra.
\end{pro}
\begin{defi}\label{pro:subadj}
Let $(A,\{\cdot,\cdots,\cdot\})$ be a $n$-pre-Lie algebra. The $n$-Lie algebra $(A,[\cdot,\cdots,\cdot]^C)$
is called the sub-adjacent $n$-Lie algebra of $(A,\{\cdot,\cdots,\cdot\})$, and denoted by $A^{c}$. $(A,\{\cdot,\cdots,\cdot\})$ is called a compatible
$n$-pre-Lie algebra of the $n$-Lie algebra $A^{c}$.
\end{defi}
Let $(A,\{\cdot,\cdots,\cdot\})$ be a $n$-pre-Lie algebra. Define the left multiplication $L: \wedge^{n-1}A\rightarrow  gl(A)$
by
\begin{equation}\label{eq:R}
L(x_1,\cdots,x_{n-1})x_n=\{x_1,\cdots,x_{n-1},x_n\},\quad \forall x_i\in A,1\leq i\leq {n}.
\end{equation}
 Moreover, we define the right multiplication $R:\otimes^{n-1}A\rightarrow gl(A)$ by
\begin{equation}
    R(x_1,\cdots,x_{n-1})x_n=\{x_n,x_1,\cdots,x_{n-1}\},\quad \forall x_i\in A,1\leq i\leq {n}.
\end{equation}
If there is a $n$-pre-Lie algebra structure on its dual
space $A^{*}$, we denote the left multiplication and right multiplication by $\mathcal{L}$ and $\mathcal{R}$ respectively.\\

 By the definitions of a $n$-pre-Lie algebra and a representation of a $n$-Lie algebra, we immediately obtain :
\begin{pro}
With the above notations, $(A,L)$ is a representation of the
$n$-Lie algebra 
$(A,[\cdot,\cdots,\cdot]^C)$. On the other hand,
let $A$ be a vector space with a $n$-linear map
$\{\cdot,\cdots,\cdot\}:(\wedge^{n-1}A)\otimes A\rightarrow A$
. Then $(A,\{\cdot,\cdots,\cdot\}) $ is a $n$-pre-Lie algebra if $[\cdot,\cdots,\cdot]^C$ defined by Eq.~\eqref{eq:ncc} is a $n$-Lie algebra and the left multiplication $L$ defined by Eq.~\eqref{eq:R}
gives a representation of this $n$-Lie algebra.
\end{pro}
\begin{pro}\label{pro:npreLieT}
Let $(A,[\cdot,\cdots,\cdot])$ be a $n$-Lie algebra and $(V,\rho)$ be  a representation. Suppose that the linear map $T:V\rightarrow A$ is an $\mathcal O$-operator associated
to $(V,\rho)$. Then there exists a $n$-pre-Lie algebra structure on $V$ given by
\begin{equation}
\{u_1,\cdots,u_n\}=\rho(Tu_1,\cdots,Tu_{n-1})u_n,\quad\forall ~ u_i\in V,1\leq i\leq n.
\end{equation}
In particular; If $V=A$, let $P:A\rightarrow A$ be a Rota-Baxter operator of weight zero associated to $(A,ad)$. Then the compatible $n$-pre-Lie algebra on $A$ is given by
\begin{equation}\label{eq:rott}
\{x_1,\cdots,x_n\}=[P(x_1),\cdots,P(x_{n-1}),x_n],
\end{equation}
for any $x_1,\cdots,x_n\in A$.
\end{pro}

\begin{cor}
With the above conditions,  $(V,[\cdot,\cdots,\cdot]^C)$ is a $n$-Lie
algebra as the sub-adjacent $n$-Lie algebra of the $n$-pre-Lie
algebra given in Proposition \ref{pro:npreLieT}, and $T$ is a $n$-Lie algebra morphism from $(V,[\cdot,\cdots,\cdot]^C)$ to $(A,[\cdot,\cdots,\cdot])$. Furthermore,
$T(V)=\{Tv\;|\;v\in V\}\subset A$ is a $n$-Lie subalgebra of $A$ and there is an induced $n$-pre-Lie algebra structure $\{\cdot,\cdots,\cdot\}_{T(V)}$ on
$T(V)$ given by
\begin{equation}
\{Tu_1,\cdots,Tu_{n}\}_{T(V)}:=T\{u_1,\cdots,u_n\},\quad\;\forall u_i\in V,1\leq i\leq n.
\end{equation}
\end{cor}

\begin{pro}\label{pro:preLieOoper}
Let $(A,[\cdot,\cdots,\cdot])$ be a $n$-Lie algebra. Then there exists a compatible $n$-pre-Lie algebra if and only if there exists an invertible $\mathcal O$-operator $T:V\rightarrow A$ associated
to a representation $(V,\rho)$. Furthermore, the compatible $n$-pre-Lie structure on $A$ is given by
\begin{equation}
\{x_1,\cdots,x_n\}_{A}=T\rho(x_1,\cdots,x_{n-1})T^{-1}(x_n),\;\forall x_i\in A,1\leq i\leq n.
\end{equation}
\end{pro}

\section{Representations of $n$-pre-Lie algebras}
In this section, we introduce the notion of a representation of a $n$-pre-Lie algebra, construct the
corresponding semi-direct product $n$-pre-Lie algebra and give the dual representation.
\begin{defi}\label{defrep}
 A  representation of a $n$-pre-Lie algebra $(A,\{\cdot,\cdots,\cdot\})$   on a vector space $V$ consists of a pair $(l,r)$, where $l:\wedge^{n-1} A \rightarrow gl(V)$ is a representation of the $n$-Lie algebra $A^c$ on $V$ and $r:(\wedge^{n-2}A)\otimes A \rightarrow gl(V)$ is a linear map such that  for all $x_1,\cdots,x_{n},y_1,\cdots,y_n\in A$, the following identities holds:
 \begin{align}
&\bullet l(x_1,\cdots,x_{n-1})r(y_1,\cdots,y_{n-1}) =r(y_1,\cdots,y_{n-1})\mu(x_1,\cdots,x_{n-1})\nonumber\\
&+\sum_{i=1}^{n-2}r(y_1,\cdots,y_{i-1},[x_1,\cdots,x_{n-1},y_i]^C,y_{i+1},\cdots,y_{n-1})+r(y_1,\cdots,y_{n-2},\{x_1,\cdots,x_{n-1},y_{n-1}\}),\\
&\bullet r([x_1,\cdots,x_n]^C,y_{1},\cdots,y_{n-2})=\sum_{i=1}^{n}(-1)^{n-i}l(x_1,\cdots,\widehat{x_i},\cdots,x_n)r(x_i,y_{1},\cdots,y_{n-2}),\\
&\bullet r(x_1,\cdots,x_{n-2},\{y_{1},\cdots,y_{n}\})=l(y_{1},\cdots,y_{n-1})r(x_1,\cdots,x_{n-2},y_{n})\nonumber\\
&+\sum_{i=1}^{n-1}(-1)^{i+1} r(y_{1},\cdots,\widehat{y_i},\cdots,y_{n})\mu(x_1,\cdots,x_{n-2},y_i),\\
&\bullet r(y_1,\cdots,y_{n-1})\mu(x_1,\cdots,x_{n-1})=l(x_1,\cdots,x_{n-1})r(y_1,\cdots,y_{n-1})\nonumber\\
&+\sum_{i=1}^{n-1}(-1)^{i}r(x_1,\cdots,\widehat{x_i},\cdots,x_{n-1},\{x_i,y_1,\cdots,y_{n-1}\}),
\end{align}
where $\quad \mu(x_1,\cdots,x_{n-1})=l(x_1,\cdots,x_{n-1})+\sum_{i=1}^{n-1}(-1)^{i}r(x_1,\cdots,\widehat{x_i},\cdots,x_{n-1},x_i)$.
\end{defi}
Let $(A,\{\cdot,\cdots,\cdot\})$ be a $n$-pre-Lie algebra and $l$ a representation of the sub-adjacent $n$-Lie algebra
$A^c$ on $V$ . Then $(l, 0)$ is a representation of the $n$-pre-Lie algebra
$(A,\{\cdot,\cdots,\cdot\})$ on the vector space $V$.
It is obvious that $(A,L,R)$ is a
representation of a $n$-pre-Lie algebra on itself, which is called the adjoint
representation.\\

By straightforward computations, we have
\begin{pro}\label{carpre}
Let $(A,\{\cdot,\cdots,\cdot\})$ be a $n$-pre-Lie algebra, $V$  a vector space and $l,r:
\otimes^{n-1}A\rightarrow  gl(V)$  two  linear
maps. Then $(V,l,r)$ is a representation of $A$ if and only if there
is a $n$-pre-Lie algebra structure $($called semi-direct product$)$
on the direct sum $A\oplus V$ of vector spaces, defined by
\begin{align}
\{x_1+u_1,\cdots,x_n+u_n\}_{A\oplus V}=&\{x_1,\cdots,x_n\}+l(x_1,\cdots,x_{n-1})(u_n)\nonumber
\\ \label{eq:sum}&+\sum_{i=1}^{n-1}(-1)^{i+1}r(x_1,\cdots,\widehat{x_i},\cdots,x_n)(u_i),
\end{align}
for $x_i\in A, u_i\in V, 1\leq i\leq n$. We denote this semi-direct product $n$-pre-Lie algebra by $A\ltimes_{l,r} V.$
\end{pro}
Let $V$ be a vector space and $(V,l,r)$ be a representation of the $n$-pre-Lie algebra $(A,\{\cdot,\cdots,\cdot\})$ on $V$. Define $\widetilde{\rho}:\wedge^{n-1}A \rightarrow gl(V)$ by
\begin{equation}
    \widetilde{\rho}(x_1,\cdots,x_{n-1})= l(x_1,\cdots,x_{n-1})+\sum_{i=1}^{n-1}(-1)^{i}r(x_1,\cdots,\widehat{x_i},\cdots,x_{n-1},x_i),
\end{equation}
for all $x_1,\cdots,x_{n-1}\in A.$
\begin{pro}\label{teald}
With the above notation, $(V,\widetilde{\rho})$ is a representation of the sub-adjacent $n$-Lie algebra $(A^c, [\cdot,\cdots, \cdot]^C)$ on the vector space $V$.
\end{pro}
\begin{proof}
By Proposition \ref{carpre}, we have the semi-direct product $n$-pre-Lie algebra $A\ltimes_{l,r} V.$ Consider its sub-adjacent $n$-Lie algebra structure
$[\cdot,\cdots, \cdot]^C$, we have for any $x_i\in A,\;u_i\in V$
\begin{align}
  &\quad \;[x_1+u_1,\cdots,x_n+u_n]_{A\oplus V}^C=\sum_{i=1}^{n}(-1)^{n-i}\{x_1+u_1,\cdots,\widehat{x_i+u_i},\cdots,x_n+u_n,x_i+u_i\}_{A\oplus V}\nonumber\\
  &=\sum_{i=1}^{n}(-1)^{n-i}\{x_1,\cdots,\widehat{x_i},\cdots,x_n,x_i\}+\sum_{i=1}^{n}(-1)^{n-i}l(x_1,\cdots,\widehat{x_i},\cdots,x_n)(u_i)\nonumber\\
  &+\sum_{i=1}^{n}(-1)^{n-i}\Big(\sum_{1 \leq i<j \leq n}(-1)^{j}r(x_1,\cdots,\widehat{x_i},\cdots,\widehat{x_j},\cdots,x_n,x_i)(u_j)\nonumber\\
  &+\sum_{1 \leq j<i \leq n}(-1)^{j+1}r(x_1,\cdots,\widehat{x_j},\cdots,\widehat{x_i},\cdots,x_n,x_i)(u_j)\Big)\nonumber\\
  &=[x_1,\cdots,x_n]^C+\sum_{i=1}^{n}(-1)^{n-i}\Big(l(x_1,\cdots,\widehat{x_i},\cdots,x_n)(u_i)+\sum_{1 \leq i<j \leq n}(-1)^{j}r(x_1,\cdots,\widehat{x_i},\cdots,\widehat{x_j},\cdots,x_n,x_i)(u_j)\nonumber\\
  &+\sum_{1 \leq j<i \leq n}(-1)^{j+1}r(x_1,\cdots,\widehat{x_j},\cdots,\widehat{x_i},\cdots,x_n,x_i)(u_j)\nonumber\\
  &=[x_1,\cdots,x_n]^C+\sum_{k=1}^{n}(-1)^{n-k}\;\widetilde{\rho}(x_1,\cdots,\widehat{x_k},\cdots,x_n)(u_k)\label{brac}.
\end{align}
By Lemma \ref{lem:sepro}, $(V,\widetilde{\rho})$ is a representation of the sub-adjacent $n$-Lie algebra $(A^c, [\cdot,\cdots, \cdot]^C)$ on the vector space $V$. The proof is finished.
\end{proof}
If $(l,r)=(L,R)$ is a representation of a $n$-pre-Lie algebra $(A,\{\cdot,\cdots,\cdot\})$, then $\widetilde{\rho}=ad$ is the adjoint representation of the sub-adjacent $n$-Lie algebra $(A^c, [\cdot,\cdots, \cdot]^C)$ on itself.
\begin{cor} Let 
$(V,l,r)$ be a representation of a $n$-pre-Lie algebra $(A,\{\cdot,\cdots,\cdot\})$ on $V$. Then the semi-product $n$-pre-Lie algebra $A\ltimes_{l,r}V$ and $A\ltimes_{\widetilde{\rho}}V$ given by the representations $(V,l,r)$ and $(V,\widetilde{\rho},0)$ respectively have the same sub-adjacent $n$-Lie algebra $A^c\ltimes_{\widetilde{\rho}}V$ given by \eqref{brac}.
\end{cor}
\begin{pro}
Let $(V,l,r)$ be a representation of a $n$-pre-Lie algebra $(A,\{\cdot,\cdots,\cdot\})$ on $V$. Then
$(V^{*},\widetilde{\rho}^{*},-r^{*})$ is a representation of the $n$-pre-Lie algebra $(A,\{\cdot,\cdots,\cdot\})$ on $V^{*}$, which is called the dual representation of the representation $(V,l,r)$.
\end{pro}
\begin{proof}
By Proposition \ref{teald}, $(V,\widetilde{\rho})$ is a representation of the sub-adjacent $n$-Lie algebra $(A^c, [\cdot,\cdots, \cdot]^C)$ on $V$. By Lemma \ref{lem:rep}, $(V^{*},\widetilde{\rho}^{*})$ is a representation of the sub-adjacent $n$-Lie algebra $(A^c, [\cdot,\cdots, \cdot]^C)$ on the dual vector space $V^{*}$. It is straightforward to deduce that other conditions in Definition \ref{defrep} also holds. We leave details to readers.
\end{proof}
\begin{cor} Let 
$(V,l,r)$ be a representation of $(A,\{\cdot,\cdots,\cdot\})$ on $V$. Then the semi-product $n$-pre-Lie algebra $A\ltimes_{l^{*},0}V^{*}$ and $A\ltimes_{\widetilde{\rho}^{*},-r^{*}}V^{*}$ given by the representations $(V^{*},l^{*},0)$ and $(V^{*},\widetilde{\rho}^{*},-r^{*})$ respectively have the same sub-adjacent $n$-Lie algebra $A^c\ltimes_{l^{*}}V^{*}$, which is the representation of the $n$-Lie algebra $(A^c, [\cdot,\cdots, \cdot]^C)$ and its representation $(V^{*},l^{*})$.
\end{cor}
\begin{defi}
Let $(A,\{\cdot,\cdots,\cdot\})$ be a $n$-pre-Lie algebra and $(V,l,r)$ be a representation. A linear operator $T:V \rightarrow A$ is called an $\mathcal O$-operator associated to $(V,l,r)$ if $T$ satisfies 
\begin{equation}\label{O-op n-pre-Lie}
    \{Tu_1,\cdots,Tu_n\}=T\Big(l(Tu_1,\cdots,Tu_{n-1})(u_n)+\sum_{i=1}^{n-1}(-1)^{i+1}r(Tu_1,\cdots,\widehat{Tu_i},\cdots,Tu_n)(u_i)\Big),
\end{equation}
$\forall u_i\in V,\;1\leq i\leq n$. If $(V,l,r)=(A,L,R)$, then $T$ is called a Rota-Baxter operator on $A$ of weight zero denoted by $P$.
\end{defi}
\begin{pro}\label{commuting rota-baxter op}
Let $(P_1,P_2)$ be a pair of commuting Rota-Baxter operators (of weight zero) on a $n$-Lie algebra $(A,[\cdot,\cdots,\cdot])$. Then $P_2$ is a Rota-Baxter operator (of weight zero) on the associated $n$-pre-Lie algebra defined by 
\begin{equation*}
    \{x_1,\cdots,x_n\}=[P_1(x_1),\cdots,P_1(x_{n-1}),x_n],\quad \forall x_i\in A,\;1\leq i\leq n.
\end{equation*}
\end{pro}
\begin{proof}
For any $x_i\in A,\;1\leq i\leq n$, we have
\begin{align*}
     &\quad \;\{P_2(x_1),\cdots,P_2(x_n)\}=[P_1(P_2(x_1)),\cdots,P_1(P_2(x_{n-1})),P_2(x_n)]\\
     &=[P_2(P_1(x_1)),\cdots,P_2(P_1(x_{n-1})),P_2(x_n)]\\
     &=P_2\Big([P_2(P_1(x_1)),\cdots,P_2(P_1(x_{n-1})),x_n]\\
     &+\sum_{i=1}^{n-1}(-1)^{n-i}[P_2(P_1(x_1)),\cdots,\widehat{P_2(P_1(x_i))},\cdots,P_2(P_1(x_{n-1})),P_1(x_i)]\Big)\\
     &=P_2\Big(\{P_2(x_1),\cdots,P_2(x_{n-1}),x_n\}\\
      &+\sum_{i=1}^{n-1}(-1)^{i+1}\{x_i,P_2(x_1),\cdots,\widehat{P_2(x_i)},\cdots,P_2(x_{n-1}),P_2(x_n)\}\Big).
\end{align*}
Hence $P_2$ is a Rota-Baxter operator (of weight zero) on the $n$-pre-Lie algebra $(A, \{\cdot, \cdots, \cdot\})$.
\end{proof}
If $(l,r)=(L,R)$ is the representation of a $n$-pre-Lie algebra $(A,\{\cdot,\cdots,\cdot\})$, then $(\widetilde{\rho}^{*},-r^{*})=(ad^{*},-R^{*})$ and the corresponding semi-direct product $n$-Lie algebra is $A^c\ltimes_{L^{*}}A^{*}$, which is the key object when we construct phase space of $n$-Lie algebras in the next section.
\section{Constructions on $n$-pre-Lie algebras}

In this section, we introduce the notion of  phase space of a $n$-Lie algebra and show that a $n$-Lie algebra has a phase space
if and only if it is sub-adjacent to a $n$-pre-Lie algebra. Moreover,
we introduce the notion of  Manin triple of $n$-pre-Lie algebras and show that there is a one-to-one
correspondence between Manin triples of $n$-pre-Lie algebras and perfect phase spaces of $n$-Lie algebras. Finally, we give a procedure to construct a $(n+1)$-pre-Lie algebra induced by a $n$-pre-Lie algebra using a trace map.
\subsection{Symplectic structures and phase spaces  of $n$-Lie algebras}
\begin{defi}\cite{Ming}
An element  $B\in \wedge^{2}A^{*}$ is a symplectic structure on a  $n$-Lie algebra $(A,[\cdot,\cdots,\cdot])$ if $B$ is a nondegenerate skew-symmetric bilinear form 
satisfying the following equality 
\begin{equation}\label{eq:symp}
B([x_1,\cdots,x_n],y)=-\sum_{i=1}^{n}(-1)^{n-i}B(x_i,[x_1,\cdots,\widehat{x_i},\cdots,x_n,y]),
\end{equation}
for any $x_i,y\in A$ and $1\leq i\leq {n}$.
\end{defi}

\begin{defi}\cite{Rui}
Let $(A, [\cdot,\cdots, \cdots])$ be a $n$-Lie algebra and let $B : A \times A \rightarrow \mathbb{K}$ be a
non-degenerate symmetric bilinear form on $A$. If $B$ satisfies
\begin{equation}
    B([x_1,\cdots,x_{n-1},x_n],x_{n+1})=-B([x_1,\cdots,x_{n-1},x_{n+1}],x_n),\;\forall x_i\in A,\;1\leq i\leq {n+1}.
\end{equation}
Then $B$ is called an invariant scalar product on $A$ or a metric on $A$, and $(A, [\cdot, \cdots, \cdot],B)$ is a
metric $n$-Lie algebra.
\end{defi}
If there exists a metric $B$ and a symplectic structure $\omega$ on a $n$-Lie algebra $(A, [\cdot, \cdots,\cdot])$, then\\ $(A, [\cdot, \cdots,\cdot],B,\omega)$ is called a metric symplectic $n$-Lie algebra.\\

Let $(A, [\cdot, \cdots,\cdot],B)$ be a metric $n$-Lie algebra, we set 
\begin{equation*}
    Der_{B}(A)=\{D\in Der(A)\;|\;B(Dx,y)+B(x,Dy)=0,\;\forall x,y\in A\}.
\end{equation*}
\begin{thm}
Let $(A,B)$ be a metric $n$-Lie algebra. Then there exists a symplectic structure
on $A$ if and only if there exists a skew-symmetric invertible derivation $D\in Der_{B}(A)$.
\end{thm}
\begin{proof}
Let $(A,B,\omega)$ be a metric symplectic $n$-Lie algebra. Defines $D : A \rightarrow A$ by
\begin{equation}\label{eq:metr}
    B(Dx,y)=\omega(x,y),\;\forall x,y\in A.
\end{equation}
It is clear that $D$ is invertible. Next we will check that $D$ is a skew-symmetric invertible
derivation of $(A, [\cdot, \cdots, \cdot],B)$. From Eq.\eqref{eq:symp}, for $x_i,y\in A,\;1\leq i\leq {n}$ 
\begin{align*}
 &\quad \;\sum_{i=1}^{n}B([x_1,\cdots,x_{i-1},D(x_i),x_{i+1}\cdots,x_n],y)-B(D([x_1,\cdots,x_n]),y)\\
 &=-\sum_{i=1}^{n}(-1)^{n-i}B([x_1,\cdots,\widehat{x}_i,\cdots,x_n,y],Dx_i)-B(D([x_1,\cdots,x_n]),y)\\
 &=-\sum_{i=1}^{n}(-1)^{n-i}\omega(x_i,[x_1,\cdots,\widehat{x}_i,\cdots,x_n,y])-\omega([x_1,\cdots,x_n],y)=0,
\end{align*}
and by the definition of $D$, we have 
\begin{equation*}
   B(Dx,y)+B(x,Dy)=\omega(x,y)+\omega(y,x)=0, 
\end{equation*}
that is, $D \in Der_{B}(A)$. Conversely, if $D \in Der_{B}(A)$ is invertible. Defines $\omega : A \times A \rightarrow \mathbb{K}$ by Eq.\eqref{eq:metr}. Then by the
above discussion, $\omega$ is non-degenerate, and satisfies Eq.\eqref{eq:symp}. The result follows.
\end{proof}
\begin{ex}
Let $(A,[\cdot,\cdots,\cdot])$ be a $n$-Lie algebra and $(m>2)$ a positive integer, we can construct a metric symplectic $n$-Lie algebra.
Let $N$ be the set of all non-negative integers
\begin{equation*}
  \mathbb{K}[t]=\{f(t)=\sum_{i=0}^{j}a_it^{i}\;|
  \;a_i\in \mathbb{K},\;j\in N\},
\end{equation*}
be the algebra of polynomials over $\mathbb{K}$.
We consider the tensor product of vector spaces 
\begin{equation}
    A_m=A\otimes (t\mathbb{K}[t]\diagup t^{m}\mathbb{K}[t]),
\end{equation}
where $t\mathbb{K}[t]\diagup t^{m}\mathbb{K}[t]$
is the quotient space of $t\mathbb{K}[t]$ module $t^{m}\mathbb{K}[t]$. Then $A_m$ is a nilpotent $n$-Lie algebra with the following bracket 
\begin{equation}
    [x_1\otimes t^{\overline{p_1}},\cdots,x_n\otimes t^{\overline{p_n}}]=[x_1,\cdots,x_n]_A\otimes t^{\overline{p_1+\cdots+p_n}},\;x_1,\cdots,x_n\in A,\;p_1,\cdots,p_n\in N\setminus \{0\}.
\end{equation}
Defines an endomorphism $D$ of $A_m$ by
\begin{equation*}
    D(x\otimes t^{\overline{p}})=p(x\otimes t^{\overline{p}}),\;\forall x\in A,\;p=1,\cdots,m-1.
\end{equation*}
Then $D$ is an invertible derivation of the $n$-Lie algebra $A_m$.
Let $\widetilde{A}_m=A_m\oplus A^{*}_m$, where 
$A^{*}_m$ is the dual space of $A_m$.
Defines the bracket
\begin{equation}
    [x_1+\xi_1,\cdots,x_n+\xi_n]_{A_m\oplus A^{*}_m}=[x_1,\cdots,x_n]_{A_m}+\sum_{k=1}^{n}(-1)^{n-k}ad^{*}(x_1,\cdots,\widehat{x_k},\cdots,x_n)(\xi_k),
\end{equation}
 and the bilinear form 
 \begin{equation}
     B(x_1+\xi_1,x_2+\xi_2)=\xi_1(x_2)+\xi_2(x_1),
 \end{equation}
 for $x_1,\cdots,x_n\in A_m,\;\xi_1,\cdots,\xi_n\in A^{*}_m$.
Then $(\widetilde{A}_m,B)$ is a metric $n$-Lie algebra. Define a  linear map $\widetilde{D}:\widetilde{A}_m \rightarrow \widetilde{A}_m$ by
\begin{equation}
 \widetilde{D}(x+\xi)=D(x)+D^{*}(\xi),\;\forall x\in A_m,\;\xi \in A^{*}_m,  
\end{equation}
where $D^{*}(\xi)=-\xi D$. Then $\widetilde{D}$ 
is an invertible derivation, and by a direct computation, we have $\widetilde{D}\in Der_{B}(A^{*}_m)$. Hence the metric $n$-Lie algebra $(A^{*}_m,B)$ admits a symplectic structure $\omega$ as follows 
\begin{equation}
    \omega(x_1+\xi_1,x_2+\xi_2)=B\Big(\widetilde{D}(x_1+\xi_1),x_2+\xi_2\Big)=-\xi_1(D(x_2))+\xi_2(D(x_1)).
\end{equation}
\end{ex}
\begin{pro}\label{pro:sympre}\cite{Ming}
Let $(A,[\cdot,\cdots,\cdot],B)$ be a symplectic $n$-Lie algebra. Then there exists a compatible $n$-pre-Lie algebra structure on $A$ given by
\begin{equation}\label{eq:sympre}
B(\{x_1,\cdots,x_n\},y)=-B(x_n,[x_1,\cdots,x_{n-1},y]),
\end{equation}
for any $x_1,\cdots,x_{n+1}\in A$.
\end{pro}
A quadratic $n$-pre-Lie algebra is a $n$-pre-Lie algebra $(A, \{\cdot, \cdots, \cdot\})$ equipped with a nondegenerate
skew-symmetric bilinear form $B\in \wedge^{2}A^{*}$ such that the following invariant condition holds:
\begin{equation}\label{invar}
B(\{x_1,\cdots,x_n\},x_{n+1})=-B(x_n,[x_1,\cdots,x_{n-1},x_{n+1}]^C),\;\forall x_1,\cdots,x_{n+1}\in A.
\end{equation}
Proposition \ref{pro:sympre} tells us that quadratic $n$-pre-Lie algebras are the underlying structures of symplectic
$n$-Lie algebras.\\

Let $V$ be a vector space and $V^{*}$ its dual space. Then there exists a natural non-degenerate skew-symmetric bilinear form $B$
on $T^{*}V=V\oplus V^{*}$ given by:
\begin{equation}\label{eq:dual2}
B(x+f,y+g)=\langle f,y\rangle-\langle g,x\rangle,\;\forall x,y\in V,\;f,g\in V^{*}.
\end{equation}
\begin{defi}
 Let $(\mathfrak{h},[\cdot,\cdots,\cdot]_\mathfrak{h})$ be a $n$-Lie algebra and $\mathfrak{h}^{*}$ its dual space.\\
 
 $\bullet$ If there is a $n$-Lie algebra structure $[\cdot,\cdots,\cdot]$
on the direct sum vector space $T^{*}\mathfrak{h}=\mathfrak{h}\oplus \mathfrak{h}^{*}$ such that $(\mathfrak{h}\oplus \mathfrak{h}^{*},[\cdot,\cdots,\cdot],B)$ is a symplectic $n$-Lie algebra, where
$B$ is given by~\eqref{eq:dual2}, and $(\mathfrak{h},[\cdot,\cdots,\cdot]_\mathfrak{h})$ and $(\mathfrak{h}^{*},[\cdot,\cdots,\cdot]_{\mathfrak{h}^{*}})$ are $n$-Lie subalgebras of $(\mathfrak{h}\oplus \mathfrak{h}^{*},[\cdot,\cdots,\cdot])$, then the symplectic $n$-Lie algebra $(\mathfrak{h}\oplus \mathfrak{h}^{*},[\cdot,\cdots,\cdot],B)$ is called a phase space of the
$n$-Lie algebra $(\mathfrak{h},[\cdot,\cdots,\cdot]_\mathfrak{h})$.\\

$\bullet$ A phase space $(\mathfrak{h} \oplus \mathfrak{h}^{*}, [\cdot, \cdots, \cdot], B)$ is called perfect if the following conditions are satisfied:
\begin{equation}\label{perfspace}
    [x_1,\cdots,x_{n-1},\alpha]\in \mathfrak{h}^{*},\quad [\alpha_1,\cdots,\alpha_{n-1},x]\in \mathfrak{h},\;\forall x,x_1,\cdots,x_{n-1}\in \mathfrak{h},\;\alpha, \alpha_1,\cdots,\alpha_{n-1}\in \mathfrak{h}^{*}.
\end{equation}
\end{defi}
Next, we show the important role of $n$-pre-Lie algebras in the study of phase spaces of $n$-Lie algebras.
\begin{thm}\label{phspace}
A $n$-Lie algebra has a phase space if and only if it is sub-adjacent to a $n$-pre-Lie algebra.
\end{thm}
\begin{proof}($\Leftarrow$) Assume $(A,\{\cdot,\cdots,\cdot\})$ is a $n$-pre-Lie algebra. By Lemma~\ref{lem:rep} and Proposition~\ref{pro:subadj}, the
left multiplication $L$ is a representation of the sub-adjacent $n$-Lie algebra $A^{c}$ on $A$, $L^{*}$ is a representation of the sub-adjacent $n$-Lie
algebra $A^{c}$ on $A^{*}$, then we have a $n$-Lie algebra structure $A^{c}\ltimes_{L^{*}}A^{*}$=$(A\oplus A^{*},[\cdot,\cdots,\cdot]_{L^{*}})$. For all $x_1,\cdots,x_{n+1}\in A$
and $\alpha_1,\cdots,\alpha_{n+1}\in A^{*}$, by Eq.\eqref{eq:dual2} we have
\begin{align*}
&B\Big([x_1+\alpha_1,\cdots,x_n+\alpha_n]_{L^{*}},x_{n+1}+\alpha_{n+1}\Big)\\
&\quad +\sum_{i=1}^{n}(-1)^{n-i}B\Big(x_i+\alpha_i,[x_1+\alpha_1,\cdots,\widehat{x_i+\alpha_i},\cdots,,x_n+\alpha_n,x_{n+1}+\alpha_{n+1}]_{L^{*}}\Big)\\
&=B\Big([x_1,\cdots,x_n]^C+\sum_{j=1}^{n}(-1)^{n-j}L^{*}(x_1,\cdots,\widehat{x_j},\cdots,x_n)\alpha_j,x_{n+1}+\alpha_{n+1}\Big)\\
&+\sum_{i=1}^{n}(-1)^{n-i}B\Big(x_i+\alpha_i,[x_1,\cdots,\widehat{x_i},\cdots,x_n,x_{n+1}]^C+\sum_{1\leq j<i\leq {n}}(-1)^{n-j}L^{*}(x_1,\cdots,\widehat{x_j},\cdots,\widehat{x_i},\cdots,x_n,x_{n+1})\alpha_j\\
&+\sum_{1\leq i<j\leq {n}}(-1)^{n-j+1}L^{*}(x_1,\cdots,\widehat{x_i},\cdots,\widehat{x_j},\cdots,x_n,x_{n+1})\alpha_j+L^{*}(x_1,\cdots,\widehat{x_i},\cdots,x_n)\alpha_{n+1}\Big)\\
&=\langle \sum_{j=1}^{n}(-1)^{n-j}L^{*}(x_1,\cdots,\widehat{x_j},\cdots,x_n)\alpha_j,x_{n+1} \rangle - \langle \alpha_{n+1},[x_1,\cdots,x_n]^C \rangle\\ 
&+\sum_{i=1}^{n}(-1)^{n-i}\Big(\langle \alpha_i,[x_1,\cdots,\widehat{x_i},\cdots,x_n,x_{n+1}]^C\rangle-\langle \sum_{1\leq j<i\leq {n}}(-1)^{n-j}L^{*}(x_1,\cdots,\widehat{x_j},\cdots,\widehat{x_i},\cdots,x_n,x_{n+1})\alpha_j\\
&+\sum_{1\leq i<j\leq {n}}(-1)^{n-j+1}L^{*}(x_1,\cdots,\widehat{x_i},\cdots,\widehat{x_j},\cdots,x_n,x_{n+1})\alpha_j+L^{*}(x_1,\cdots,\widehat{x_i},\cdots,x_n)\alpha_{n+1} ,x_i\rangle \Big).
\end{align*}
And by Eq.~\eqref{eq:dual}, we have
\begin{align*}
&B\Big([x_1+\alpha_1,\cdots,x_n+\alpha_n]_{L^{*}},x_{n+1}+\alpha_{n+1}\Big)\\
&\quad +\sum_{i=1}^{n}(-1)^{n-i}B\Big(x_i+\alpha_i,[x_1+\alpha_1,\cdots,\widehat{x_i+\alpha_i},\cdots,,x_n+\alpha_n,x_{n+1}+\alpha_{n+1}]_{L^{*}}\Big)\\
&=-\langle \alpha_j,\sum_{j=1}^{n}(-1)^{n-j}L(x_1,\cdots,\widehat{x_j},\cdots,x_n)x_{n+1} \rangle - \langle \alpha_{n+1},[x_1,\cdots,x_n]^C \rangle\\ 
&+\sum_{i=1}^{n}(-1)^{n-i}\Big(\langle \alpha_i,[x_1,\cdots,\widehat{x_i},\cdots,x_n,x_{n+1}]^C\rangle+\langle \alpha_j, \sum_{1\leq j<i\leq {n}}(-1)^{n-j}L(x_1,\cdots,\widehat{x_j},\cdots,\widehat{x_i},\cdots,x_n,x_{n+1})x_i\\
&+\sum_{1\leq i<j\leq {n}}(-1)^{n-j+1}L(x_1,\cdots,\widehat{x_i},\cdots,\widehat{x_j},\cdots,x_n,x_{n+1})x_i\rangle+\langle \alpha_{n+1},L(x_1,\cdots,\widehat{x_i},\cdots,x_n)x_i\rangle \Big)\\
&=-\langle \alpha_j,\sum_{j=1}^{n}(-1)^{n-j}\{x_1,\cdots,\widehat{x_j},\cdots,x_n,x_{n+1}\} \rangle - \langle \alpha_{n+1},[x_1,\cdots,x_n]^C \rangle\\ 
&+\sum_{i=1}^{n}(-1)^{n-i}\Big(\langle \alpha_i,[x_1,\cdots,\widehat{x_i},\cdots,x_n,x_{n+1}]^C\rangle+\langle \alpha_j, \sum_{1\leq j<i\leq {n}}(-1)^{n-j}\{x_1,\cdots,\widehat{x_j},\cdots,\widehat{x_i},\cdots,x_n,x_{n+1},x_i\}\\
&+\sum_{1\leq i<j\leq {n}}(-1)^{n-j+1}\{x_1,\cdots,\widehat{x_i},\cdots,\widehat{x_j},\cdots,x_n,x_{n+1},x_i\}\rangle+\langle \alpha_{n+1},\{x_1,\cdots,\widehat{x_i},\cdots,x_n,x_i\}\rangle \Big)=0.
\end{align*}
Since $\{x_{\sigma(1)},\cdots,x_{\sigma(n-1)},x_n\}=Sgn(\sigma)\{x_1,\cdots,x_{n-1},x_n\}$ and by Eq.\eqref{eq:ncc}, we deduce that $B$ is a symplectic structure on
the semi-direct product $n$-Lie algebra $A^{c}\ltimes_{L^{*}}A^{*}$. Moreover, $(A^{c},[\cdot,\cdots,\cdot]^{C})$ is  a subalgebra of $A^{c}\ltimes_{L^{*}}A^{*}$ and $A^{*}$ is an abelian subalgebra of $A^{c}\ltimes_{L^{*}}A^{*}$. Thus, the symplectic $n$-Lie algebra $(A^{c}\ltimes_{L^{*}}A^{*},B)$ is a phase space of the sub-adjacent $n$-Lie algebra $(A^{c},[\cdot,\cdots,\cdot]^{C})$.\\
($\Rightarrow$) Let $(T^{*}\mathfrak{h}=\mathfrak{h}\oplus \mathfrak{h}^{*},[\cdot,\cdots,\cdot],B)$ be a phase space of a $n$-Lie algebra $(\mathfrak{h},[\cdot,\cdots,\cdot]_{\mathfrak{h}})$. By Proposition
\ref{pro:sympre}, there exists a compatible $n$-pre-Lie algebra structure $\{\cdot,\cdots,\cdot\}$ on $T^{*}\mathfrak{h}$ given by Eq.~\eqref{eq:sympre}.
Since $(\mathfrak{h},[\cdot,\cdots,\cdot]_\mathfrak{h})$ is a subalgebra of $(\mathfrak{h}\oplus \mathfrak{h}^{*},[\cdot,\cdots,\cdot])$, we have
\begin{align*}
&B(\{x_1,\cdots,x_n\},x_{n+1})=-B(x_n,[x_1,\cdots,x_{n-1},x_{n+1}])\\
&-B(x_n,[x_1,\cdots,x_{n-1},x_{n+1}]_\mathfrak{h})=0,\;\forall x_1,\cdots,x_{n+1}\in \mathfrak{h}.
\end{align*}
Thus, $\{x_1,\cdots,x_n\}\in \mathfrak{h}$, which implies that $(\mathfrak{h},\{\cdot,\cdots,\cdot\}_\mathfrak{h})$ is a subalgebra of the $n$-pre-Lie algebra $(\mathfrak{h}\oplus \mathfrak{h}^{*},\{\cdot,\cdots,\cdot\})$. Its sub-adjacent $n$-Lie algebra $(\mathfrak{h}^{c},[\cdot,\cdots,\cdot]^{C})$ is exactly the original $n$-Lie algebra
$(\mathfrak{h},[\cdot,\cdots,\cdot]_\mathfrak{h})$.
\end{proof}
\begin{cor}\label{suball}
If $(\mathfrak{h}\oplus \mathfrak{h}^{*},[\cdot,\cdots,\cdot],B)$ is a phase space of a $n$-Lie algebra $(\mathfrak{h},[\cdot,\cdots,\cdot]_\mathfrak{h})$ and $(\mathfrak{h}\oplus \mathfrak{h}^{*},\{\cdot,\cdots,\cdot\})$ is 
the associated $n$-pre-Lie algebra. Then both $(\mathfrak{h},\{\cdot,\cdots,\cdot\}_\mathfrak{h})$ and $( \mathfrak{h}^{*},\{\cdot,\cdots,\cdot\}_{\mathfrak{h}^{*}})$ are subalgebras of the
$n$-pre-Lie algebra $(\mathfrak{h}\oplus \mathfrak{h}^{*},\{\cdot,\cdots,\cdot\})$.
\end{cor}

\begin{pro}
Let $\rho$ be a representation of $(\mathfrak{h},[\cdot,\cdots,\cdot]_\mathfrak{h})$ on $\mathfrak{h}$ and $\rho^{*}$ is its dual representation. If $(\mathfrak{h}\oplus \mathfrak{h}^{*},[\cdot,\cdots,\cdot],B)$ is a phase space of a $n$-Lie algebra $(\mathfrak{h},[\cdot,\cdots,\cdot]_\mathfrak{h})$ such that the $n$-Lie algebra $(\mathfrak{h}\oplus \mathfrak{h}^{*},[\cdot,\cdots,\cdot])$ is a semi-direct 
product $\mathfrak{h}\ltimes_{\rho^{*}}\mathfrak{h}^{*}$, then 
\begin{equation}
\{x_1,\cdots,x_n\}=\rho(x_1,\cdots,x_{n-1})x_n,\;\forall x_1,\cdots,x_n\in \mathfrak{h}.
\end{equation} 
defines a $n$-pre-Lie algebra structure on $\mathfrak{h}$.
\end{pro}
\begin{proof} For all $x_1,\cdots,x_n\in \mathfrak{h}$, $\alpha\in \mathfrak{h}^{*}$, we have 
\begin{align*}
&\langle \alpha,\{x_1,\cdots,x_n\}\rangle=-B(\{x_1,\cdots,x_n\},\alpha)=B(x_n,[x_1,\cdots,x_{n-1},\alpha]_{\mathfrak{h}\oplus \mathfrak{h}^{*}})\\
&B(x_n,\rho^{*}(x_1,\cdots,x_{n-1})\alpha)=-\langle \rho^{*}(x_1,\cdots,x_{n-1})\alpha,x_n\rangle=\langle \alpha,\rho(x_1,\cdots,x_{n-1})x_n\rangle.
\end{align*}
Therefore, $\{x_1,\cdots,x_n\}=\rho(x_1,\cdots,x_{n-1})x_n$.
\end{proof}
\begin{ex}
Let $(A,\{\cdot,\cdots,\cdot\}_A$ be a $n$-pre-Lie algebra. Since there is a semi-direct product $n$-pre-Lie algebra $(A\ltimes_{L^{*},0}A^{*},\{\cdot,\cdots,\cdot\}_{L^{*},0})$ on the phase space $T^{*}A^c=A^c\ltimes_{L^{*}}A^{*}$, one can construct a new phase space $T^{*}A^c \ltimes_{L^{*}}(T^{*}A^c)^{*}$. This process can be continued indefinitely. Hence,
there exist a series of phase spaces $\{A_{(m)}\}_{m\geq 2}:$ 
\begin{equation*}
   A_{(1)}= A^c,\;\;\;A_{(2)}=T^{*}A_{(1)}=A^c\ltimes_{L^{*}}A^{*},\cdots,\;\;\;A_{(m)}=T^{*}A_{(m-1)},\cdots.
\end{equation*}
$A_{(m)}\; (m\geq 2)$ is called the symplectic double of $A_{(m-1)}$.
\end{ex}
At the end of this section, we introduce the notion of  Manin triple of $n$-pre-Lie algebras.
\begin{defi}
A Manin triple of $n$-pre-Lie algebras is a triple $(\mathcal{A};A,A')$, where \\
$\bullet$ $(\mathcal{A},\{\cdot,\cdots,\cdot\},B)$ is a quadratic $n$-pre-Lie algebra,\\
$\bullet$ both $A$ and $A'$ are isotropic subalgebras of $(\mathcal{A},\{\cdot,\cdots,\cdot\})$,\\
$\bullet$ $\mathcal{A}=A \oplus A'$ as vector spaces,\\
$\bullet$ for all $x,x_1,\cdots,x_{n-1}\in A$ and $\alpha,\alpha_1,\cdots,\alpha_{n-1}\in A'$, holds:
\begin{equation}\label{condmanin}
    \{x_1,\cdots,x_{n-1},\alpha\}\in A',\; \{\alpha,x_1,\cdots,x_{n-1}\}\in A',\; \{\alpha_1,\cdots,\alpha_{n-1},x\}\in A,\; \{x,\alpha_1,\cdots,\alpha_{n-1}\}\in A. 
\end{equation}
\end{defi}
In a Manin triple of $n$-pre-Lie algebras, since the skew-symmetric bilinear form $B$ is nondegenerate,
$A'$ can be identified with $A^{*}$ via
\begin{equation*}
  \langle \alpha,x\rangle \triangleq B(\alpha,x),\;\forall x\in A,\; \alpha \in A'.  
\end{equation*}
Thus, $\mathcal{A}$ is isomorphic to $A\oplus A^{*}$  naturally and the bilinear form $B$ is exactly given by \eqref{eq:dual2}. By
the invariant condition \eqref{invar}, we can obtain the precise form of the $n$-pre-Lie structure $\{\cdot, \cdots,\cdot\}$ on
$A \oplus A^{*}$.
\begin{pro}\label{promanin}
Let $(A\oplus A^{*};A,A^{*})$ be a Manin triple of $n$-pre-Lie algebras, where the nondegenerate
skew-symmetric bilinear form $B$ on the $n$-pre-Lie algebra is given by \eqref{eq:dual2}. Then we
have
\begin{align}
 &\label{manin1}\{x_1,\cdots,x_{n-1},\alpha\}=\Big(L^{*}(x_1,\cdots,x_{n-1})+\sum_{i=1}^{n-1}(-1)^{i}R^{*}(x_1,\cdots,\widehat{x_i},\cdots,x_{n-1},x_i)\Big)\alpha,\\
  &\label{manin2}\{\alpha,x_1,\cdots,x_{n-1}\}=-R^{*}(x_1,\cdots,x_{n-1})\alpha,\\
   &\label{manin3}\{\alpha_1,\cdots,\alpha_{n-1},x\}= \Big(\mathcal{L}^{*}(\alpha_1,\cdots,\alpha_{n-1})+\sum_{i=1}^{n-1}(-1)^{i}\mathcal{R}^{*}(\alpha_1,\cdots,\widehat{\alpha_i},\cdots,\alpha_{n-1},\alpha_i)\Big)x,\\
     &\label{manin4}\{x,\alpha_1,\cdots,\alpha_{n-1}\}=-\mathcal{R}^{*}(\alpha_1,\cdots,\alpha_{n-1})x.
\end{align}
\end{pro}
\begin{proof}
$\forall x_1,\cdots,x_n \in A,\; \alpha \in A^{*}$, we have
\begin{align*}
  &\quad \;\langle\{x_1,\cdots,x_{n-1},\alpha\},x_n\rangle=B( \{x_1,\cdots,x_{n-1},\alpha\},x_n)\\
   &=-B(\alpha,[x_1,\cdots,x_n]^{C})=-B(\alpha,\sum_{i=1}^{n}(-1)^{n-i}\{x_1,\cdots,\widehat{x_i},\cdots,x_n,x_i\})\\
    &=-B(\alpha,\Big(L(x_1,\cdots,x_{n-1})+\sum_{i=1}^{n-1}(-1)^{i}R(x_1,\cdots,\widehat{x_i},\cdots,x_{n-1},x_i)\Big)x_n)\\
   &=-\langle \alpha,\Big(L(x_1,\cdots,x_{n-1})+\sum_{i=1}^{n-1}(-1)^{i}R(x_1,\cdots,\widehat{x_i},\cdots,x_{n-1},x_i)\Big)x_n \rangle\\
   &=\langle \Big(L^{*}(x_1,\cdots,x_{n-1})+\sum_{i=1}^{n-1}(-1)^{i}R^{*}(x_1,\cdots,\widehat{x_i},\cdots,x_{n-1},x_i)\Big)\alpha,x_n \rangle,
\end{align*}
which implies that \eqref{manin1} holds. We have
\begin{align*}
   &\quad \;\langle \{\alpha,x_1,\cdots,x_{n-1}\},x_n\rangle =B( \{\alpha,x_1,\cdots,x_{n-1}\},x_n)=-B(x_{n-1},[\alpha,x_1,\cdots,x_{n-2},x_n]^{C})\\
   &=B(x_{n-1},[x_n,x_1,\cdots,x_{n-2},\alpha]^{C})=-B(\{x_n,x_1,\cdots,x_{n-1}\},\alpha)=B(\alpha,\{x_n,x_1,\cdots,x_{n-1}\})\\
      &=\langle \alpha,R(x_1,\cdots,x_{n-1})x_n \rangle =-\langle R^{*}(x_1,\cdots,x_{n-1})\alpha,x_n \rangle,
\end{align*}
which implies that \eqref{manin2} holds. Similarly, we can deduce that \eqref{manin3} and \eqref{manin4} hold.
\end{proof}
\begin{thm}
There is a one-to-one correspondence between Manin triples of $n$-pre-Lie algebras
and perfect phase spaces of $n$-Lie algebras. More precisely, if $(A\oplus A^{*};A,A^{*})$ is a Manin triple
of $n$-pre-Lie algebras, then $(A \oplus A^{*}, [\cdot, \cdots, \cdot]^{C},B)$ is a symplectic $n$-Lie algebra, where $B$ is given by
\eqref{eq:dual2}. Conversely, if $(\mathfrak{h} \oplus \mathfrak{h}^{*}, [\cdot, \cdots, \cdot],B)$ is a perfect phase space of a $n$-Lie algebra $(\mathfrak{h}, [\cdot, \cdots, \cdot]_\mathfrak{h})$, then
$(\mathfrak{h}\oplus \mathfrak{h}^{*}; \mathfrak{h}, \mathfrak{h}^{*})$ is a Manin triple of $n$-pre-Lie algebras, where the $n$-pre-Lie algebra structure on $\mathfrak{h} \oplus \mathfrak{h}^{*}$
is given by \eqref{eq:sympre}.
\end{thm}
\begin{proof}($\Rightarrow$)
Let $(A\oplus A^{*};A,A^{*})$ be a Manin triple of $n$-pre-Lie algebras. Denote by $\{\cdot,\cdots, \cdot\}_A$ and $\{\cdot,\cdots,\cdot\}_{A^{*}}$
the $n$-pre-Lie algebra structure on $A$ and $A^{*}$ respectively, and denote by $[\cdot, \cdots, \cdot]_A$ and $[\cdot, \cdots, \cdot]_{A^{*}}$ the
corresponding sub-adjacent $n$-Lie algebra structure on $A$ and $A^{*}$ respectively. By Proposition \ref{promanin},
it is straightforward to deduce that the corresponding $n$-Lie algebra structure $[\cdot, \cdots, \cdot]^{C}$ on $A\oplus A^{*}$ is
given by
\begin{align*}
   & [x_1+\alpha_1,\cdots,x_n+\alpha_n]^{C }=[x_1,\cdots,x_n]_A+\sum_{i=1}^{n}(-1)^{n-i} \mathcal{L}^{*}(\alpha_1,\cdots,\widehat{\alpha_i},\cdots,\alpha_n)(x_i)\\
   &+[\alpha_1,\cdots,\alpha_n]_{{A}^{*}}+\sum_{i=1}^{n}(-1)^{n-i} L^{*}(x_1,\cdots,\widehat{x_i},\cdots,x_n)(\alpha_i).
\end{align*}
For all $x_1,\cdots,x_{n+1}\in A$ and $\alpha_1,\cdots,\alpha_{n+1}\in A^{*}$, similarly to  the Proof of Theorem \ref{phspace}, we have
\begin{align*}
    &B\Big([x_1+\alpha_1,\cdots,x_n+\alpha_n]^C,x_{n+1}+\alpha_{n+1}\Big)\\
&\quad +\sum_{i=1}^{n}(-1)^{n-i}B\Big(x_i+\alpha_i,[x_1+\alpha_1,\cdots,\widehat{x_i+\alpha_i},\cdots,,x_n+\alpha_n,x_{n+1}+\alpha_{n+1}]^C\Big)=0.
\end{align*}

Since \begin{equation*}
    \{x_1,\cdots,x_n\}_A= Sgn(\sigma)\{x_{\sigma(1)},\cdots,x_{\sigma(n-1)},x_n\},\;\forall \sigma \in S_{n},\;and\;\sigma(n)=n
\end{equation*}
and \begin{equation*}
    \{\alpha_1,\cdots,\alpha_n\}_{A^{*}}= Sgn(\sigma)\{\alpha_{\sigma(1)},\cdots,\alpha_{\sigma(n-1)},\alpha_n\},\;\forall \sigma \in S_{n},\;and\;\sigma(n)=n,
\end{equation*}
then we deduce that $B$ is a
symplectic structure on the $n$-Lie algebra $(A \oplus A^{*}, [\cdot,\cdots, \cdot]^{C})$. Therefore, it is a phase space.\\
($\Leftarrow$) Conversely, let $(\mathfrak{h}\oplus \mathfrak{h}^{*}, [\cdot, \cdots, \cdot], B)$ be a phase space of the $n$-Lie algebra $(\mathfrak{h}, [\cdot, \cdots, \cdot]_{\mathfrak{h}})$. By Proposition
\ref{pro:sympre}, there exists a $n$-pre-Lie algebra structure $\{\cdot, \cdots, \cdot\}$ on $\mathfrak{h} \oplus \mathfrak{h}^{*}$ given by \eqref{eq:sympre} such that $(\mathfrak{h}\oplus
\mathfrak{h}^{*}, \{\cdot, \cdots, \cdot\},B)$ is a quadratic $n$-pre-Lie algebra. By Corollary \ref{suball}, $(\mathfrak{h}, \{\cdot, \cdots, \cdot\}_{\mathfrak{h}})$ and $(\mathfrak{h}^{*}, \{\cdot,\cdots, \cdot\}_{\mathfrak{h}^{*}})$
are $n$-pre-Lie subalgebras of $(\mathfrak{h}\oplus \mathfrak{h}^{*}, \{\cdot,\cdots,\cdot\})$. It is obvious that both $\mathfrak{h}$ and $\mathfrak{h}^{*}$ are isotropic. Thus,
we only need to show that \eqref{condmanin} holds. By \eqref{perfspace}, for all $x_1,\cdots,x_{n-1}\in \mathfrak{h}$ and $\alpha_1,\cdots,\alpha_{n-1}\in \mathfrak{h}^{*}$, we have
\begin{equation*}
    B(\{x_1,\cdots,x_{n-1},\alpha_1\})=-B(\alpha_1,[x_1,\cdots,x_{n-1},\alpha_2]^{C})=0,
\end{equation*}
which implies that $\{x_1,\cdots,x_{n-1},\alpha_1\}\in \mathfrak{h}^{*}$. Similarly, we can show that the other conditions in \eqref{condmanin} also
hold. The proof is finished.
\end{proof}

\subsection{$(n+1)$-pre-Lie algebras induced by $n$-pre-Lie algebras}
In this subsection we show how to derive $(n+1)$-pre-Lie algebras from  $n$-pre-Lie algebras and a trace map, extending to pre-Lie structures the construction on Lie structures provided in \cite{Makhlouf2}.
 Let us start
by defining a map which is skew-symmetric on the first $n$ variables that will be used to induce the higher order
algebra.

\begin{defi}
 Let $\phi:\otimes^n A \rightarrow A$ be a $n$-linear map and let $\tau$ be a map from $A$
to $\mathbb{K}$.\\ Define $\phi_{\tau}:\otimes^{n+1}A \rightarrow A$ by
\begin{equation}\label{eq:induced}
\phi_{\tau}(x_1,\cdots,x_n,x_{n+1})=\sum_{k=1}^{n}(-1)^{k+1}\tau(x_k)\phi(x_1,\cdots,\widehat{x_k},\cdots,x_n,x_{n+1})
\end{equation}
where the hat over $\widehat{x_k}$ on the right hand side means that $x_k$ is excluded, that is $\phi$
is calculated on $(x_1, \cdots , x_{k-1}, x_{k+1}, \cdots , x_{n+1})$.
\end{defi}

\begin{defi}
Let $\phi:\otimes^n A \rightarrow A$ be a $n$-linear map. A linear map $\tau : A\rightarrow \mathbb{K}$ is called a $\phi$-trace if it 
satisfies for all $x_1,\cdots,x_n\in A$
\begin{equation*}
\tau(\phi(x_1,\cdots,x_n))=0,\forall x_1,\cdots,x_n\in A.
\end{equation*}
\end{defi}

\begin{lem}\label{skewphi}
 Let $\phi:(\wedge^{n-1}A)\otimes A \rightarrow A$  be a n-linear map and $\tau$
a linear map $A \rightarrow \mathbb{K}$. Then $\phi_{\tau}$ is a $(n + 1)$-linear  map such that is skew-symmetric on the first $n$ variables.
Furthermore, if $\tau$ is a $\phi$-trace then $\tau$ is a $\phi_{\tau}$-trace.
\end{lem}
\begin{proof}
 The $(n+1)$-linearity property of $\phi_{\tau}$ follows from $n$-linearity of $\phi$ and linearity
of $\tau$ as it is a linear combination of $(n + 1)$-linear maps. 
To prove the skew-symmetry on the first $n$ variables one simply notes that
\begin{equation*}
  \phi_{\tau}(x_1,\cdots,x_n,x_{n+1})=\frac{1}{(n-1)!}\sum_{\sigma \in S_n}Sgn(\sigma)\tau(x_{\sigma(1)})\phi(x_{\sigma(2)},\cdots,x_{\sigma(n)},x_{n+1}),  
\end{equation*}
which is clearly skew-symmetric on the first $n$ variables. Since each term in $\phi_{\tau}$ is proportional to $\phi$ and
since $\tau$ is linear and a $\phi$-trace, $\tau$ will also be a $\phi_{\tau}$-trace.
\end{proof}
\begin{thm}\label{theindu}
Suppose that $(A,\phi)$ is a $n$-pre-Lie algebra and $\tau:A\rightarrow \mathbb{K}$ is  a $\phi$-trace, then $(A, \phi_{\tau})$  is a $(n + 1)$-pre-Lie algebra. We shall say that $(A,\phi_{\tau})$ is induced by  $(A,\phi)$.
\end{thm}
\begin{proof}
Since $\phi_{\tau}$ is skew-symmetric on the first $n$ variables and multilinear by Lemma \ref{skewphi}, one only has to check that
the two identity of $(n+1)$-pre-Lie algebra $(A,\phi_{\tau})$ are fulfilled. The first identity is written as
\begin{align*}
    &\phi_{\tau}\Big(x_1,\cdots,x_n,\phi_{\tau}(y_1,\cdots,y_{n+1})\Big)-\phi_{\tau}\Big(y_1,\cdots,y_n,\phi_{\tau}(x_1,\cdots,x_n,y_{n+1})\Big)\\
    &-\sum_{i=1}^{n}\phi_{\tau}\Big(y_1,\cdots,y_{i-1},[x_1,\cdots,x_n,y_i]^C,y_{i+1},\cdots,y_n,y_{n+1}\Big),
\end{align*}
where $[x_1,\cdots,x_{n+1}]^C=\sum_{i=1}^{n+1}(-1)^{(n+1)-i}\phi_{\tau}(x_1,\cdots,\widehat{x_i},\cdots,x_{n+1},x_i)$.\\
Let us write the left-hand-side of this equation as $B-C-D$. Furthermore, we expand $B$ into terms $B_{kl}$ and $C$ into terms $C_{lk}$ such that
\begin{align*}
    &B_{kl}=(-1)^{k+l}\tau(x_k)\tau(y_l)\phi\Big(x_1,\cdots,\widehat{x_k},\cdots,x_n,\phi(y_1,\cdots,\widehat{y_l},\cdots,y_n,y_{n+1})\Big),\\
    &B=\sum_{k=1}^{n}\sum_{l=1}^{n}B_{kl},\\
    &C_{lk}=(-1)^{l+k}\tau(y_l)\tau(x_k)\phi\Big(y_1,\cdots,\widehat{y_l},\cdots,y_n,\phi(x_1,\cdots,\widehat{x_k},\cdots,x_n,y_{n+1})\Big),\\
    &C=\sum_{l=1}^{n}\sum_{k=1}^{n}C_{lk}.
\end{align*}
Taking into account that $\tau$ is a $\phi-$trace form, we expand $D$ as
\begin{align*}
 &D=\sum_{i=1}^{n}\phi_{\tau}\Big(y_1,\cdots,[x_1,\cdots,x_n,y_i]^C,\cdots,y_n,y_{n+1}\Big)\\
 &=\sum_{i=1}^{n}\phi_{\tau}\Big(y_1,\cdots,\sum_{j=1}^{n}(-1)^{(n+1)-j}\phi_{\tau}(x_1,\cdots,\widehat{x_j},\cdots,x_n,y_i,x_j)+\phi_{\tau}(x_1,\cdots,x_n,y_i),\cdots,y_n,y_{n+1}\Big)\\
 &=\sum_{i=1}^{n}\phi_{\tau}\Big(y_1,\cdots,\phi_{\tau}(x_1,\cdots,x_n,y_i),\cdots,y_n,y_{n+1}\Big)\\
 &+\sum_{i=1}^{n}\sum_{j=1}^{n}(-1)^{(n+1)-j}\phi_{\tau}\Big(y_1,\cdots,\phi_{\tau}(x_1,\cdots,\widehat{x_j},\cdots,x_n,y_i,x_j),\cdots,y_n,y_{n+1}\Big)\\
 &=D^{(1)}+D^{(2)},
\end{align*} where
\begin{align*}
    &D^{(1)}=\sum_{i=1}^{n}\sum_{k=1}^{n}(-1)^{k+1}\tau(x_k)\phi_{\tau}\Big(y_1,\cdots,\phi(x_1,\cdots,\widehat{x_k},\cdots,x_n,y_i),\cdots,y_n,y_{n+1}\Big),\\
    &D^{(2)}=\sum_{i=1}^{n}\sum_{j=1}^{n}(-1)^{(n+1)-j}\sum_{k=1,k\neq j}^{n}(-1)^{k+1}\tau(x_k)\phi_{\tau}\Big(y_1,\cdots,\phi(x_1,\cdots,\widehat{x_k},\cdots,\widehat{x_j},\cdots,x_n,y_i,x_j),\cdots,y_n,y_{n+1}\Big)\\
    &+(-1)^{n+1}\sum_{i=1}^{n}\sum_{j=1}^{n}(-1)^{(n+1)-j}\tau(y_i)\phi_{\tau}\Big(y_1,\cdots,\phi(x_1,\cdots,\widehat{x_j},\cdots,x_n,x_j),\cdots,y_n,y_{n+1}\Big).
\end{align*}
Let us now consider the expansion of $D^{(1)}$ via $D^{(1)}_{kl}$
\begin{align*}
&D^{(1)}_{kl}=\sum_{i=1}^{l-1}(-1)^{k+l}\tau(x_k)\tau(y_l)\phi\Big(y_1,\cdots,\phi(x_1,\cdots,\widehat{x_k},\cdots,x_n,y_i),\cdots,\widehat{y_l},\cdots,y_n,y_{n+1}\Big)\\
&+\sum_{i=l+1}^{n}(-1)^{k+l}\tau(x_k)\tau(y_l)\phi\Big(y_1,\cdots,\widehat{y_l},\cdots,\phi(x_1,\cdots,\widehat{x_k},\cdots,x_n,y_i),\cdots,y_n,y_{n+1}\Big)\\
&D^{(1)}=\sum_{i=1}^{n}\sum_{k=1}^{n}\sum_{l=1}^{n}D^{(1)}_{kl},
\end{align*}
and we expand $D^{(2)}$ as
\begin{align*}
 &D^{(2)}=\sum_{i=1}^{n}\sum_{j=1}^{n}(-1)^{(n+1)-j}\sum_{k=1,k\neq j}^{n}\sum_{l=1}^{n}\\
 &\times\Big(\sum_{i=1}^{l-1}(-1)^{k+l}\tau(x_k)\tau(y_l)\phi\Big(y_1,\cdots,\phi(x_1,\cdots,\widehat{x_k},\cdots,\widehat{x_j},\cdots,x_n,y_i,x_j),\cdots,\widehat{y_l},\cdots,y_n,y_{n+1}\Big)\\
 &+\sum_{i=l+1}^{n}(-1)^{k+l}\tau(x_k)\tau(y_l)\phi\Big(y_1,\cdots,\widehat{y_l},\cdots,\phi(x_1,\cdots,\widehat{x_k},\cdots,\widehat{x_j},\cdots,x_n,y_i,x_j),\cdots,y_n,y_{n+1}\Big)\Big)\\
 &+\sum_{i=1}^{n}\sum_{j=1}^{n}(-1)^{(n+1)-j}\sum_{l=1}^{n}\Big(\sum_{i=1}^{l-1}(-1)^{n+l}\tau(y_i)\tau(y_l)\phi\Big(y_1,\cdots,\phi(x_1,\cdots,\widehat{x_j},\cdots,x_n,x_j),\cdots,\widehat{y_l},\cdots,y_n,y_{n+1}\Big)\\
 &+\sum_{i=l+1}^{n}(-1)^{n+l}\tau(y_i)\tau(y_l)\phi\Big(y_1,\cdots,\widehat{y_l},\cdots,\phi(x_1,\cdots,\widehat{x_j},\cdots,x_n,x_j),\cdots,y_n,y_{n+1}\Big)\Big).
\end{align*}
By a direct simplification and using the two identity of $n$-pre-Lie algebras $(A,\phi)$, we get that the first identity of $(n+1)$-pre-Lie algebras is satisfied for $\phi_{\tau}$. Also, by the same way we can prove that $\phi_{\tau}$ satisfies the second identity of $(n+1)$-pre-Lie algebras.
\end{proof}
In particular, one may derive from a pre-Lie algebra  a 3-pre-Lie algebra using a trace map. 
Recall that  a pre-Lie algebra is a pair $(A,\circ)$, where $A$ is a vector space and $\circ : A\times A\rightarrow A$
is a bilinear multiplication satisfying, for all $x, y, z\in A$,
\begin{equation}\label{eq:pree}
(x\circ y)\circ z-x\circ(y\circ z)=(y\circ x)\circ z-y\circ (x\circ z).
\end{equation}
It is obvious that all associative algebras are pre-Lie
algebras and for a pre-Lie algebra $(A,\circ)$, the commutator $[x,y]^{C}=x\circ y-y\circ x$,
defines a Lie algebra. By Theorem \ref{theindu},
suppose that $(A,\circ)$ is a pre-Lie algebra and $\tau:A\rightarrow \mathbb{K}$ is a linear map, such that $\tau$ is a $\circ$-trace. Then $(A,\{\cdot,\cdot,\cdot\}_{\tau})$ is a $3$-pre-Lie algebra, and we say that it is induced by $(A,\circ)$ where
$\{\cdot,\cdot,\cdot\}_{\tau}:(\wedge^2 A)\otimes A\rightarrow A$ is defined by
\begin{equation}
\{x,y,z\}_{\tau}=\tau(x)y\circ z-\tau(y)x\circ z,\;\forall x,y,z\in A.
\end{equation}

\begin{ex}
Let $(A,\circ )$ be a three-dimensional pre-Lie algebra with a basis
$\{e_1, e_2, e_3\}$ whose nonzero brackets are given as follows:

    $$e_3 \circ e_2 =e_2,\quad e_3 \circ e_3 =- e_3.$$
  We define a trace map as
\begin{equation*}
    \tau(e_1)=a,\;  \tau(e_2)=0,\; \tau(e_3)=0,
\end{equation*}
for any $a \in  \mathbb{K}$. Then, according to Theorem \ref{theindu}, we obtain a three-dimensional $3$-pre-Lie algebra with a basis
$\{e_1, e_2, e_3\}$ whose nonzero brackets are given as follows:
\begin{equation*}
    \{e_1,e_3,e_2\}=a e_2,\quad \{e_1,e_3,e_3\}=-a e_3.
\end{equation*}
\end{ex}
Since $\tau$ is also a $\phi_{\tau}$-trace, one can repeat the procedure in Theorem \ref{theindu} to induce
a $(n + 2)$-pre-Lie algebra from an $n$-pre-Lie algebra. However,
the result is an abelian algebra.
\begin{pro}
Let $\mathcal{A}=(A,\phi)$ be a $n$-pre-Lie algebra. Let $\mathcal{A}^{'}$ 
be any $(n+ 1)$-pre-Lie algebra induced by $\mathcal{A}$ via the $\phi$-trace $\tau$. If
$\mathcal{A}^{''}$ is a $(n + 2)$-pre-Lie algebra induced by $\mathcal{A}^{'}$ using the same $\tau$ again, then $\mathcal{A}^{''}$ is abelian.
\end{pro}
\begin{proof}
By the definition of $\phi_{\tau}$ , the bracket on the algebra $\mathcal{A}^{''}$ can be written as
\begin{equation*}
    (\phi_{\tau})_{\tau}=\sum_{k=1}^{k+1}(-1)^{k+1}\tau(x_k)\phi_{\tau}(x_1,\cdots,\widehat{x_k},\cdots,x_{n+1},x_{n+2})
\end{equation*}
Expanding the bracket $\phi_{\tau}$ , there will be, for every choice of integers $k < l$, two
terms which are proportional to $\tau(x_k)\tau(x_l)$. For $l=1,\cdots,n+1$, their sum becomes 
\begin{equation*}
  \tau(x_k)\tau(x_l)\phi( x_1,\cdots,\widehat{x_k},\cdots,\widehat{x_l},\cdots,x_{n+1},x_{n+2})\Big((-1)^{k+l}+(-1)^{k+l-1}\Big)=0.
\end{equation*}
Hence $  (\phi_{\tau})_{\tau}(x_1,\cdots,x_{n+2})=0$, for $x_1,\cdots,x_{n+2}\in A$.
\end{proof}
\begin{rmk}
Note that one might choose a different $\phi_{\tau}$-trace when repeating the construction,
and this can lead to non-abelian algebras.
\end{rmk}
\begin{pro}
Let $(V,l,r)$ be a representation of a $n$-pre-Lie algebra $(A,\{\c,\cdots,\c\})$ and $\tau$ be a
trace map of $A$. Then $(V,l_{\tau},r_{\tau})$ is a representation of the $(n+1)$-pre-Lie algebra $(A,\{\cdot,\cdots,\cdot\}_{\tau})$ where
$l_{\tau}: \wedge^{n} A \rightarrow gl(V)$, $r_{\tau}: \otimes^{n}A \rightarrow gl(V) $ are defined by
\begin{align}
    &l_{\tau}(x_1,\cdots,x_n)=\sum_{k=1}^{n}(-1)^{k+1}\tau(x_k)l(x_1,\cdots,\widehat{x_k},\cdots,x_n),\\
    &r_{\tau}(x_1,\cdots,x_n)=\sum_{k=1}^{n-1}(-1)^{k}\tau(x_k)r(x_1,\cdots,\widehat{x_k},\cdots,x_{n-1},x_n),\quad \forall x_i\in A.
\end{align}
\end{pro}
\begin{proof}
Let $\tau$ be a trace map of $(A,\{\c,\cdots,\c\} )$ and $(V,l,r)$ be a representation of $A$. We define $\overline{\tau}:A\oplus V \rightarrow \mathbb{K}$ by
\begin{equation*}
   \overline{\tau}(x+u)=\tau(x),\; \forall x\in A,\;\forall u\in V.
\end{equation*}
Then $\overline{\tau}$ is a trace map on the semidirect product $n$-pre-Lie algebra $A \ltimes_{(l,r)} V $. Then by \eqref{eq:induced} , there is a $(n+1)$-pre-Lie algebra structure $\{\cdot,\cdots,\cdot\}_{\overline{\tau}}$ on the vector space $A\oplus V $ given by
\begin{align*}
   &\{x_1+u_1,\cdots,x_{n+1}+u_{n+1}\}_{\overline{\tau}}=\sum_{k=1}^{n}(-1)^{k+1}\overline{\tau}(x_k+u_k)\{x_1+u_1,\cdots,\widehat{x_k+u_k},\cdots,x_n+u_n,x_{n+1}+u_{n+1}\}_{A\oplus V}\\
   &=\sum_{k=1}^{n}(-1)^{k+1}\tau(x_k)\Big(\{x_1,\cdots,\widehat{x_k},\cdots,x_n,x_{n+1}\}+l(x_1,\cdots,\widehat{x_k},\cdots, x_n)u_{n+1}\\
   &+\sum_{1\leq k < j \leq n}(-1)^{j}r(x_1,\cdots,\widehat{x_k},\cdots,\widehat{x_j},\cdots,x_n,x_{n+1})u_j+\sum_{1\leq j < k \leq n}(-1)^{j+1}r(x_1,\cdots,\widehat{x_j},\cdots,\widehat{x_k},\cdots,x_n,x_{n+1})u_j\\
   &=\{x_1,\cdots,x_{n+1}\}_{\tau}+l_{\tau}(x_1,\cdots,x_n)u_{n+1}+\sum_{i=1}^{n}(-1)^{i+1}r_{\tau}(x_1,\cdots,\widehat{x_i},\cdots,x_n,x_{n+1})(u_i).
\end{align*}
Since a semidirect product of $(n+1)$-pre-Lie algebras is equivalent to  representations of a $(n+1)$-pre-Lie algebras.
Thus, we deduce that $(V,l_{\tau},r_{\tau})$ is a representation of the $(n+1)$-pre-Lie algebra $(A,\{\cdot,\cdots,\cdot\}_{\tau})$.
\end{proof}
\begin{lem}
Let $(A,\{\c,\cdots,\c\} )$ be a $n$-pre-Lie algebra and  $\tau$ be a trace map. Let $D:A \rightarrow A$ be a $n$-pre-Lie derivation, then $\tau \circ D$ is a trace map.
\end{lem}
\begin{proof}
For any $x_i\in A$, we have
\begin{equation*}
    \tau(D(\{x_1,\cdots,x_n\}))=\tau\Big(\sum_{i=1}^{n}\{x_1,\cdots,D(x_i),\cdots,x_n\}\Big)=\sum_{i=1}^{n}\tau(\{x_1,\cdots,D(x_i),\cdots,x_n\})=0.
\end{equation*}
\end{proof}
\begin{thm}
Let $D:A \rightarrow A$ be a derivation of a $n$-pre-Lie algebra $A$, then $D$ is a derivation
of the induced $(n+1)$-pre-Lie algebra if and only if:
\begin{equation*}
    \{x_1,\cdots,x_{n+1}\}_{\tau \circ D}=0,\quad \forall x_i\in A.
\end{equation*}
\end{thm}
\begin{proof}
Let $D$ be a derivation of $A$ and $x,y,z\in A$, we obtain
\begin{align*}
  &D(\{x_1,\cdots,x_{n+1}\}_{\tau})= D(\sum_{k=1}^{n}(-1)^{k+1}\tau(x_k)\{x_1,\cdots,\widehat{x_k},\cdots,x_n,x_{n+1}\})\\
  &\quad \quad \quad \quad \quad \quad \quad \;\; =\sum_{k=1}^{n}(-1)^{k+1}\tau(x_k)\Big(\sum_{i=1,i\neq k}^{n+1}\{x_1,\cdots,\widehat{x_k},\cdots,Dx_i,\cdots,x_n,x_{n+1}\}\Big)\\
  &\quad \quad \quad \quad \quad \quad \;\; \quad+\sum_{k=1}^{n}(-1)^{k+1}(\tau\circ D)(x_k)\{x_1,\cdots,\widehat{x_k},\cdots,x_n,x_{n+1}\}-\{x_1,\cdots,x_{n+1}\}_{\tau \circ D}\\
  &\quad \quad \quad \quad \quad \quad \;\; \quad=\sum_{k=1}^{n+1}\{x_1,\cdots,Dx_k,\cdots,x_{n+1}\}_\tau-\{x_1,\cdots,x_{n+1}\}_{\tau \circ D}.
\end{align*}
\end{proof}

\begin{rmk}This construction fits also with sub-adjacent $n$-Lie algebras. 
The sub-adjacent $(n+1)$-Lie algebra of the induced $(n+1)$-pre-Lie algebra is the induced $(n+1)$-Lie algebra of the sub-adjacent $n$-Lie algebra of the $n$-pre-Lie algebra with respect to the same trace map.
\end{rmk}

\section{$n$-L-dendriform algebras}

In this section, we introduce the notion of a $n$-L-dendriform algebra which is exactly the $n$-ary version
of a L-dendriform algebra. We provide some construction results in terms of $\mathcal{O}$-operator and symplectic
structure.
\begin{defi}
Let $A$ be a vector space with two $n$-linear maps $\nwarrow: (\wedge^{n-1} A)\otimes A \rightarrow A$ and $\nearrow:A\otimes (\wedge^{n-2}A)\otimes A \rightarrow A$. The tuple $(A,\nwarrow,\nearrow)$ is called a
$n$-L-dendriform algebra if the following identities hold
\begin{align}
 \label{3-L-dendriform1}&   \nwarrow(x_1,\cdots,x_{n-1},\nwarrow(y_1,\cdots,y_{n}))-  \nwarrow(y_1,\cdots,y_{n-1},\nwarrow(x_1,\cdots,x_{n-1},y_{n})) \nonumber\\
&\hspace{2 cm} =\sum_{i=1}^{n-1}\nwarrow(y_1,\cdots,y_{i-1},[x_1,\cdots,x_{n-1},y_i]^C,y_{i+1},\cdots,y_{n})      ,\\
 &  \label{3-L-dendriform2}\nwarrow([x_1,\cdots,x_n]^C,y_{1},\cdots,y_{n-1})=\sum_{i=1}^{n}(-1)^{n-i}\nwarrow(x_1,\cdots,\widehat{x_i},\cdots,x_n,\nwarrow(x_i,y_{1},\cdots,y_{n-1})),
   \\
 &\label{3-L-dendriform3}  \nwarrow(x_1,\cdots,x_{n-1}, \nearrow(y_{n},y_1,\cdots,y_{n-1}))- \nearrow(y_{n},y_1,\cdots,y_{n-2},\{x_1,\cdots,x_{n-1},y_{n-1}\}^h) \nonumber\\
& \hspace{2 cm} =\nearrow(\{x_1,\cdots,x_{n-1},y_{n}\}^v,y_1,\cdots,y_{n-1}) \nonumber\\ &\hspace{2 cm}+\sum_{i=1}^{n-2}\nearrow(y_{n},y_1,\cdots,y_{i-1},[x_1,\cdots,x_{n-1},y_i]^C,y_{i+1},\cdots,y_{n-1}), \\
 &   \label{3-L-dendriform4}     \nearrow(y_{n-1},[x_1,\cdots,x_n]^C,y_{1},\cdots,y_{n-2})\nonumber\\
 &\hspace{2 cm}=\sum_{i=1}^{n}(-1)^{n-i}\nwarrow(x_1,\cdots,\widehat{x_i},\cdots,x_n, \nearrow(y_{n-1},x_i,y_{1},\cdots,y_{n-2})) ,\\
 & \label{3-L-dendriform5} \nearrow(x_{n-1},x_1,\cdots,x_{n-2},\{y_{1},\cdots,y_{n}\}^h)-\nwarrow(y_{1},\cdots,y_{n-1}, \nearrow(x_{n-1},x_1,\cdots,x_{n-2},y_{n}))
 \nonumber\\
&\hspace{2 cm}=\sum_{i=1}^{n-1}(-1)^{i+1} \nearrow(\{x_1,\cdots,x_{n-2},y_i,x_{n-1}\}^v,y_{1},\cdots,\widehat{y_i},\cdots,y_{n})
      ,\\
 &   \label{3-L-dendriform6} \nearrow(\{x_1,\cdots,x_{n-1},y_{n}\}^v,y_1,\cdots,y_{n-1})-  \nwarrow(x_1,\cdots,x_{n-1}, \nearrow(y_{n},y_1,\cdots,y_{n-1}))\nonumber\\
&\hspace{2 cm} =\sum_{i=1}^{n-1}(-1)^{i} \nearrow(y_{n},x_1,\cdots,\widehat{x_i},\cdots,x_{n-1},\{x_i,y_1,\cdots,y_{n-1}\}^h)     ,
\end{align}
 for all $x_i,y_i \in A$, $1\leq i \leq n$, where
\begin{align}
& \{x_1,\cdots,x_n\}^h=\nwarrow(x_1,\cdots,x_n)+ \sum_{i=1}^{n-1}(-1)^{i+1}\nearrow(x_i,x_1,\cdots,\widehat{x_i},\cdots,x_n)  ,   \label{accolade horizintal}\\
& \{x_1,\cdots,x_n\}^v=\nwarrow(x_1,\cdots,x_n)+ \sum_{i=1}^{n-1}(-1)^{i}\nearrow(x_n,x_1,\cdots,\widehat{x_i},\cdots,x_{n-1},x_i)  ,   \label{accolade vertical}\\
& [x_1,\cdots,x_n]^C=\sum_{i=1}^{n}(-1)^{n-i}\{x_1,\cdots,\widehat{x_i},\cdots,x_n,x_i\}^h=\sum_{i=1}^{n}(-1)^{n-i}\{x_1,\cdots,\widehat{x_i},\cdots,x_n,x_i\}^v, \label{crochet}
\end{align}
for any $x_i \in A$, $1\leq i \leq n$.
\end{defi}

\begin{rmk}
Let $(A,\nwarrow, \nearrow)$ be a $n$-L-dendriform algebra.  If $ \nearrow=0$, then $(A,\nwarrow)$ is a $n$-pre-Lie algebra.
\end{rmk}

\begin{pro}\label{3LDendTo3PreLie}
Let $(A,\nwarrow, \nearrow)$ be a $n$-L-dendriform algebra.
\begin{enumerate}
  \item  The bracket given in \eqref{accolade horizintal} defines a $n$-pre-Lie algebra structure on $A$ which is called the associated  horizontal  $n$-pre-Lie algebra of  $(A,\nwarrow, \nearrow)$  and $(A,\nwarrow, \nearrow)$ is also called  a compatible $n$-L-dendriform algebra structure  on the  $n$-pre-Lie algebra $(A,\{\c,\cdots,\c\}^h)$.
  \item  The bracket given in \eqref{accolade vertical}  defines a $n$-pre-Lie algebra structure on $A$ which is called the associated vertical  $n$-pre-Lie algebra of  $(A,\nwarrow, \nearrow)$  and $(A,\nwarrow, \nearrow)$ is also called  a compatible $n$-L-dendriform algebra structure  on the  $n$-pre-Lie algebra $(A,\{\c,\cdots,\c\}^v)$.
\end{enumerate}
\end{pro}

\begin{proof}
We will just prove item 1.
Note, first that  by the skew-symmetry of the first $n- 1$ variables of $\nwarrow$ and since $\nearrow(x_1,x_{\sigma(2)},\cdots,x_{\sigma(n-1)},x_n)=Sgn(\sigma)\nearrow(x_1,x_2,\cdots,x_{n-1},x_n),\; \forall \sigma \in S_{n}$, with  $\sigma(1)=1$ and $\sigma(n)=n$, the induced $n$-bracket $\{\c,\cdots,\c\}^h$ given by Eq.
\eqref{accolade horizintal} is skew-symmetric on the first $n-1$ variables.  \\
Let $x_i,y_i \in A,\ 1 \leq i \leq n$. Then
\begin{align*}
&\{x_1,\cdots,x_{n-1},\{y_1,\cdots,y_{n}\}^h\}^h- \{y_1,\cdots,y_{n-1},\{x_1,\cdots,x_{n-1},y_{n}\}^h\}^h\\
&-\sum_{i=1}^{n-1}\{y_1,\cdots,y_{i-1},[x_1,\cdots,x_{n-1},y_i]^C,y_{i+1},\cdots,y_{n}\}^h\\
&=\nwarrow(x_1,\cdots,x_{n-1},\nwarrow(y_1,\cdots,y_n)+\sum_{i=1}^{n-1}(-1)^{i+1}\nearrow(y_i,y_1,\cdots,\widehat{y_i},\cdots,y_{n-1},y_n))\\
&+\sum_{i=1}^{n-1}(-1)^{i+1}\nearrow(x_i,x_1,\cdots,\widehat{x_i},\cdots,x_{n-1},\{y_1,\cdots,y_n\}^h)\\
&-\nwarrow(y_1,\cdots,y_{n-1},\nwarrow(x_1,\cdots,x_{n-1},y_n)+\sum_{i=1}^{n-1}(-1)^{i+1}\nearrow(x_i,x_1,\cdots,\widehat{x_i},\cdots,x_{n-1},y_n))\\
&-\sum_{i=1}^{n-1}(-1)^{i+1}\nearrow(y_i,y_1,\cdots,\widehat{y_i},\cdots,y_{n-1},\{x_1,\cdots,x_{n-1},y_n\}^h)\\
&-\sum_{i=1}^{n-1}\Big(\nwarrow(y_1,\cdots,y_{i-1},[x_1,\cdots,x_{n-1},y_i]^C,y_{i+1},\cdots,y_{n-1},y_n)\\
&+\sum_{j=1,j\neq i}^{n-1}(-1)^{j+1}\nearrow(y_j,y_1,\cdots,\widehat{y_j},\cdots,y_{i-1},[x_1,\cdots,x_{n-1},y_i]^C,y_{i+1},\cdots,y_{n-1},y_n)\Big)\\
&-\sum_{i=1}^{n-1}(-1)^{i+1}\nearrow(\sum_{k=1}^{n-1}\{x_1,\cdots,\widehat{x_k},\cdots,x_{n-1},y_i,x_k\}^v+\{x_1,\cdots,x_{n-1},y_i\}^v,y_1,\cdots,\widehat{y_i},\cdots,y_{n-1},y_n).
\end{align*}
From identities \eqref{3-L-dendriform1},\eqref{3-L-dendriform3} and \eqref{3-L-dendriform5}, we obtain immediately that \eqref{n-pre-Lie 1} hold.
On the other hand, by the same way we can prove the second identity \eqref{n-pre-Lie 2} of $n$-pre-Lie algebra structure.
\end{proof}

\begin{cor}
Let $(A,\nwarrow, \nearrow)$ be a $n$-L-dendriform algebra. Then the bracket defined in  \eqref{crochet} defines a $n$-Lie algebra structure on $A$ which is called the associated   $n$-Lie algebra of  $(A,\nwarrow, \nearrow)$.
\end{cor}
The following Proposition is obvious.
\begin{pro}
Let $(A,\nwarrow, \nearrow)$ be a $n$-L-dendriform algebra. Define $L_{\nwarrow},R_{ \nearrow}: \otimes^{n-1} A \to gl(A)$ and $\rho:\wedge^{n-1} \rightarrow gl(A)$ by
\begin{equation*}
 L_{\nwarrow}(x_1,\cdots,x_{n-1})x_n=\nwarrow(x_1,\cdots,x_{n-1},x_n),\quad  R_{ \nearrow}(x_1,\cdots,x_{n-1})x_n= \nearrow(x_n,x_1,\cdots,x_{n-1}),\end{equation*} 
\begin{equation*}
    \rho(x_1,\cdots,x_{n-1})x_n=\nwarrow(x_1,\cdots,x_n)+ \sum_{i=1}^{n-1}(-1)^{i}\nearrow(x_n,x_1,\cdots,\widehat{x_i},\cdots,x_{n-1},x_i)\end{equation*}
 for all $x_i \in A$, $1\leq i \leq n.$
Then
\begin{enumerate}
\item[(1)] $(A,L_{\nwarrow},R_{ \nearrow})$ is a representation of  its  horizontal associated   $n$-pre-Lie algebra $(A,\{\c,\cdots,\c\}^h)$  and $(A,L_{\nwarrow},-L_{ \nearrow})$ is a representation of  its  vertical associated   $n$-pre-Lie algebra $(A,\{\c,\cdots,\c\}^v)$.
\item[(2)]  $(A,L_{\nwarrow})$ is a representation of its associated $n$-Lie algebra $(A^c,[\c,\cdots,\c]^C)$.
\item[(3)] $(A,\rho)$ is a representation of its associated $n$-Lie algebra $(A^c,[\c,\cdots,\c]^C)$.
\end{enumerate}
\end{pro}

\begin{rmk}
In the sense of the above conclusion (1), a $n$-L-dendriform algebra is understood as
a $n$-ary  algebra structure whose left and right multiplications give a bimodule structure on the underlying vector space of the $n$-pre-Lie algebra defined by certain commutators. It can be regarded as the \textbf{rule} of introducing the notion of $n$-L-dendriform algebra, which more generally, is the \textbf{rule}
of introducing the notions of $n$-pre-Lie algebras,  the Loday algebras and their Lie, Jordan and alternative  algebraic analogues.
\end{rmk}

\begin{thm}
 Let $(A,\{\c,\cdots,\c\})$ be a $n$-pre-Lie algebra and $(V,l,r)$ be a representation. Suppose that  $T: V \to A$ is an $\mathcal{O}$-operator associated to $(V,l,r)$. Then there exists a $n$-L-dendriform algebra structure on $V$ given by
 \begin{align}
  \nwarrow(u_1,\cdots,u_n)=l(Tu_1,\cdots,Tu_{n-1})u_n  , \quad  \nearrow(u_1,\cdots,u_n)= r(Tu_2,\cdots,Tu_n)u_1, 
 \end{align}
 $\forall u_i \in V,\;1\leq i \leq n.$
 Therefore, there exists two associated $n$-pre-Lie algebra structures on $V$ and $T$ is a homomorphism of $n$-pre-Lie algebras. Moreover, $T(V)=\{T(v)| v \in V \}$ is a $n$-pre-Lie subalgebra of $(A,\{\c,\cdots,\c\})$ and there is an induced $n$-L-dendriform algebra structure on $T(V)$ given by
  \begin{align}
  \nwarrow(Tu_1,\cdots,Tu_n)=T(  \nwarrow(u_1,\cdots,u_n))  , \quad  \nearrow(Tu_1,\cdots,Tu_n)= T( \nearrow(u_1,\cdots,u_n)),
 \end{align}
$ \forall u_i \in V,\;1\leq i \leq n.$
\end{thm}

\begin{proof}
Let $u_i \in V$,$1\leq i \leq n$ .  Define $\{\c,\cdots,\c\}_V^h, \{\c,\cdots,\c\}_V^v,[\c,\cdots,\c]^C_V:\otimes^n V \to V$ by
\begin{align*}
& \{u_1,\cdots,u_n\}_V^h=\nwarrow(u_1,\cdots,u_n)+ \sum_{i=1}^{n-1}(-1)^{i+1}\nearrow(u_i,u_1,\cdots,\widehat{u_i},\cdots,u_n),\\
& \{u_1,\cdots,u_n\}_V^v=\nwarrow(u_1,\cdots,u_n)+ \sum_{i=1}^{n-1}(-1)^{i}\nearrow(u_n,u_1,\cdots,\widehat{u_i},\cdots,u_{n-1},u_i),\\
&[u_1,\cdots,u_n]^C_V=\sum_{i=1}^{n}(-1)^{n-i}\{u_1,\cdots,\widehat{u_i},\cdots,u_n,u_i\}_V^h.
\end{align*}
Using identity \eqref{O-op n-pre-Lie}, we have
\begin{align*}
T\{u_1,\cdots,u_n\}_V^h &=T(\nwarrow(u_1,\cdots,u_n)+ \sum_{i=1}^{n-1}(-1)^{i+1}\nearrow(u_i,u_1,\cdots,\widehat{u_i},\cdots,u_n))\\
&=T(l(Tu_1,\cdots,Tu_{n-1})(u_n)+\sum_{i=1}^{n-1}(-1)^{i+1}r(Tu_1,\cdots,\widehat{Tu_i},\cdots,Tu_n)(u_i))\\
&=\{Tu_1,\cdots,Tu_n\}^h
\end{align*}
and
\begin{align*}
&T[u_1,\cdots,u_n]^C_V=T( \sum_{i=1}^{n}(-1)^{n-i}\{u_1,\cdots,\widehat{u_i},\cdots,u_n,u_i\}_V^h)\\
&=\sum_{i=1}^{n}(-1)^{n-i}\{Tu_1,\cdots,\widehat{Tu_i},\cdots,Tu_n,Tu_i\}^h=[Tu_1,\cdots,Tu_n]^C.
\end{align*}
It is straightforward that
$$\nwarrow(u_{\sigma(1)},\cdots,u_{\sigma(n-1)},u_n)=  l(Tu_{\sigma(1)},\cdots,Tu_{\sigma(n-1)})u_n=Sgn(\sigma)\nwarrow(u_1,\cdots,u_{n-1},u_n),$$
$\forall \sigma \in S_{n-1}$. Furthermore, for any $u_i,v_i \in V,\ 1 \leq i \leq n$, we have
\begin{align*}
&\nwarrow(u_1,\cdots,u_{n-1},\nwarrow(v_1,\cdots,v_{n}))-  \nwarrow(v_1,\cdots,v_{n-1},\nwarrow(u_1,\cdots,u_{n-1},v_{n})) \\
&-\sum_{i=1}^{n-1}\nwarrow(v_1,\cdots,v_{i-1},[u_1,\cdots,u_{n-1},v_i]_V^C,v_{i+1},\cdots,v_{n}),\\
&= l(Tu_1,\cdots,Tu_{n-1})l(Tv_1,\cdots,Tv_{n-1})uv{n}-  l(Tv_1,\cdots,Tv_{n-1})l(Tu_1,\cdots,Tu_{n-1})v_{n} \\
&-\sum_{i=1}^{n-1}l(Tv_1,\cdots,Tv_{i-1},T[u_1,\cdots,u_{n-1},v_i]_V^C,Tv_{i+1},\cdots,Tv_{n-1})v_{n}\\
&=\Big( [l(Tu_1,\cdots,Tu_{n-1}),l(Tv_1,\cdots,Tv_{n-1})]\\
&-\sum_{i=1}^{n-1}l(Tv_1,\cdots,Tv_{i-1},[Tu_1,\cdots,Tu_{n-1},Tv_i]^C,Tv_{i+1},\cdots,Tv_{n-1})\Big)v_{n}\\
&=0.
\end{align*}
This implies that \eqref{3-L-dendriform1} holds.
Moreover, \eqref{3-L-dendriform3} holds. Indeed,
\begin{align*}
&\nwarrow(u_1,\cdots,u_{n-1}, \nearrow(v_{n},v_1,\cdots,v_{n-1}))- \nearrow(v_{n},v_1,\cdots,v_{n-2},\{u_1,\cdots,u_{n-1},v_{n-1}\}_V^h) \nonumber\\
&-\nearrow(\{u_1,\cdots,u_{n-1},v_{n}\}_V^v,v_1,\cdots,v_{n-1}) \nonumber\\ 
&-\sum_{i=1}^{n-2}\nearrow(v_{n},v_1,\cdots,v_{i-1},[u_1,\cdots,u_{n-1},v_i]_V^C,v_{i+1},\cdots,v_{n-1}), \\
&=l(Tu_1,\cdots,Tu_{n-1})r(Tv_1,\cdots,Tv_{n-1})v_{n}-r(Tv_1,\cdots,Tv_{n-2},T\{u_1,\cdots,u_{n-1},v_{n-1}\}^h_V)v_{n}\\
&-r(Tv_1,\cdots,Tv_{n-1})\{u_1,\cdots,u_{n-1},v_{n}\}^v_V \\
&-\sum_{i=1}^{n-2}r(Tv_1,\cdots,Tv_{i-1},T[u_1,\cdots,u_{n-1},v_i]_V^C,Tv_{i+1},\cdots,Tv_{n-1})v_{n}\\
&=l(Tu_1,\cdots,Tu_{n-1})r(Tv_1,\cdots,Tv_{n-1})v_{n}-r(Tv_1,\cdots,Tv_{n-2},\{Tu_1,\cdots,Tu_{n-1},Tv_{n-1}\}^h)v_{n}\\
&-r(Tv_1,\cdots,Tv_{n-1})\mu(Tu_1,\cdots,Tu_{n-1})v_{n} \\
&-\sum_{i=1}^{n-2}r(Tv_1,\cdots,Tv_{i-1},[Tu_1,\cdots,Tu_{n-1},Tv_i]^C,Tv_{i+1},\cdots,Tv_{n-1})v_{n}\\
&=0.
\end{align*}
The other conclusions follow immediately by the same way.
\end{proof}
\begin{cor}\label{CorLDendViaRB}
   Let $(A,\{\c,\cdots,\c\})$ be a $n$-pre-Lie algebra and  $P: A \to A$ is a Rota-Baxter operator of weight $0$. Then there exists a $n$-L-dendriform algebra structure on $A$ given by
 \begin{align}\label{LDendViaRB}
  \nwarrow(x_1,\cdots,x_n)=\{P(x_1),\cdots,P(x_{n-1}),x_n\}  , \quad  \nearrow(x_1,\cdots,x_n)= \{x_1,P(x_2),\cdots,P(x_{n})\},  \end{align}
for all $x_i \in A$,$\;1\leq i \leq n.$
\end{cor}

\begin{pro}\label{3Lden by invert O-op}
 Let $(A,\{\c,\cdots,\c\})$ be a $n$-pre-Lie algebra.  Then there exists a compatible
$n$-L-dendriform  algebra if and only if there exists an invertible $\mathcal{O}$-operator on $A$.
\end{pro}
 \begin{proof}
 Let $T$ be an invertible $\mathcal{O}$-operator of $A$ associated to a representation $(V, l,r)$.
Then there exists a $n$-L-dendriform algebra structure on $V$ defined by
 \begin{align*}
  \nwarrow(u_1,\cdots,u_n)=l(Tu_1,\cdots,Tu_{n-1})u_n  , \quad  \nearrow(u_1,\cdots,u_n)= r(Tu_2,\cdots,Tu_n)u_1, \forall u_i \in V.
 \end{align*}
In addition there exists a $n$-L-dendriform algebra structure on $T(V)=A$ given by
 \begin{align*}
  \nwarrow(Tu_1,\cdots,Tu_n)=T(l(Tu_1,\cdots,Tu_{n-1})u_n)  , \quad  \nearrow(Tu_1,\cdots,Tu_n)= T(r(Tu_2,\cdots,Tu_n)u_1), \forall u_i \in V.
 \end{align*}
 If we put $x_1=Tu_1,\cdots,x_n=Tu_n$, we get
  \begin{align*}
  \nwarrow(x_1,\cdots,x_n)=T(l(x_1,\cdots,x_{n-1})T^{-1}(x_n))  , \quad  \nearrow(x_1,\cdots,x_n)= T(r(x_2,\cdots,x_n)T^{-1}(x_1)), \forall x_i \in A.
 \end{align*}
It is a compatible $n$-L-dendriform algebra structure on $A$. Indeed,
\begin{align*}
&  \nwarrow(x_1,\cdots,x_n)+ \sum_{i=1}^{n-1}(-1)^{i+1}\nearrow(x_i,x_1,\cdots,\widehat{x_i},\cdots,x_n) \\
&= T(l(x_1,\cdots,x_{n-1})T^{-1}(x_n)) +\sum_{i=1}^{n-1}(-1)^{i+1} T(r(x_1,\cdots,\widehat{x_i},\cdots,x_n)T^{-1}(x_i)) \\
&= \{TT^{-1}(x_1),\cdots,TT^{-1}(x_n)\}=\{x_1,\cdots,x_n\}.
\end{align*}
Conversely,   let $(A,\{\c,\cdots,\c\})$ be a $n$-pre-Lie algebra and $(A,\nwarrow, \nearrow)$ its compatible $n$-L-dendriform algebra.  Then the identity map $id: A \to A$ is an $\mathcal{O}$-operator of  $(A,\{\c,\cdots,\c\})$ associated to $(A,L,R)$.
 \end{proof}

\begin{defi}
Let $(A,\{\c,\cdots,\c\})$ be a $n$-pre-Lie algebra and $B$ be a symmetric bilinear form on $A$.  We say that $B$ is closed if it satisfies
\begin{align}\label{symplectic bilinear form}
B(\{x_1,\cdots,x_n\},w)+B(x_n,[x_1,\cdots,x_{n-1},w]^C)\nonumber\\
-\sum_{i=1}^{n-1}(-1)^{i+1}B(x_i,\{w,x_1,\cdots,\widehat{x_i},\cdots,x_{n-1},x_n\})=0,
\end{align}
for any $x_i,w \in A$, $\;1\leq i \leq n$.  If in addition $B$ is non-degenerate, then $B$ is called a pseudo-Hessian structure on $A$.

A $n$-pre-Lie algebra $(A,\{\c,\cdots,\c\})$  is called pseudo-Hessian  if it is
equipped with a  symmetric, closed and  nondegenerate form $B$.  It is
denoted by  $(A,\{\c,\cdots,\c\},B)$.
\end{defi}

\begin{pro}
Let $(A,\{\c,\cdots,\c\},B)$ be a pseudo-Hessian $n$-pre-Lie algebra. Then there exists a compatible $n$-L-dendriform algebra structure on $A$ given by
\begin{align}\label{3Ldend by  form}
&B(\nwarrow(x_1,\cdots,x_n),w)=-B(x_n,[x_1,\cdots,x_{n-1},w]^C),\nonumber\\
&B( \nearrow(x_1,\cdots,x_n),w)=B(x_1,\{w,x_2,\cdots,x_n\}), 
\end{align}
$\forall x_i,w \in A,\;1\leq i \leq n$.
\end{pro}

\begin{proof}
Define the linear map $T: A^* \to A$ by $\langle T^{-1}x,y\rangle=B(x,y)$.  Using Eq. \eqref{symplectic bilinear form}, we obtain that $T$ is an invertible $\mathcal{O}$-operator  on $A$ associated to the dual representation $(A^*, ad^*,-R^*)$.   By Proposition \ref{3Lden by invert O-op}, there exists a compatible $n$-L-dendriform algebra structure given by
  \begin{align*}
  \nwarrow(x_1,\cdots,x_n)=T(ad^*(x_1,\cdots,x_{n-1})T^{-1}(x_n))  , \quad  \nearrow(x_1,\cdots,x_n)= -T(R^*(x_2,\cdots,x_n)T^{-1}(x_1)).
 \end{align*}
Then we have
\begin{align*}
   &B(\nwarrow(x_1,\cdots,x_n),w)=  B(T(ad^*(x_1,\cdots,x_{n-1})T^{-1}(x_n)) ,w)=\langle ad^*(x_1,\cdots,x_{n-1})T^{-1}(x_n),w\rangle \\
    &\quad \quad \quad \quad \quad \quad \quad \quad\;\; =-\langle T^{-1}(x_n),[x_1,\cdots,x_{n-1},w]^C\rangle=-B(x_n,[x_1,\cdots,x_{n-1},w]^C)
\end{align*}
and
\begin{align*}
    &B( \nearrow(x_1,\cdots,x_n),w)=  -B(T(R^*(x_2,\cdots,x_n)T^{-1}(x_1))) ,w)=-\langle R^*(x_2,\cdots,x_n)T^{-1}(x_1),w\rangle \\
    &\quad \quad \quad \quad \quad \quad \quad \quad \;= \langle T^{-1}(x_1),\{w,x_2,\cdots,x_n\}\rangle=B(x_1,\{w,x_2,\cdots,x_n\}).
\end{align*}
The proof is finished.
\end{proof}

\begin{cor}
Let $(A,\{\c,\cdots,\c\},B)$ be a pseudo-Hessian $n$-pre-Lie algebra and let $(A^c,[\c,\cdots,\c]^C)$ be its associated $n$-Lie algebra.  Then there exists a $n$-pre-Lie algebraic structure $(A,\{\c,\cdots,\c\}')$ on $A$ given by
\begin{align}
B(\{x_1,\cdots,x_n\}',w)=-B(x_n,[x_1,\cdots,x_{n-1},w]^C)+\sum_{i=1}^{n-1}(-1)^iB(x_n,\{w,x_1,\cdots,\widehat{x_i},\cdots,x_{n-1},x_i\}).
\end{align}
\end{cor}

\begin{pro}
Let $\{P_1,P_2\}$ be a pair of  commuting Rota-Baxter operators (of weight zero) on a $n$-Lie algebra $(A, [\c,\cdots,\c])$. Then there exists a $n$-L-dendriform algebra structure on $A$ defined by
\begin{align*}
&\nwarrow(x_1,\cdots,x_n)=[P_1P_2(x_1),\cdots,P_1P_2(x_{n-1}),x_n],\\ &\nearrow(x_1,\cdots,x_n)=[P_1(x),P_1P_2(x_2),\cdots,P_1P_2(x_{n-1}),P_2(x_n)], \forall x_i \in A.
\end{align*}
\end{pro}

\begin{proof}
It follows immediately from Proposition \ref{commuting rota-baxter op} and Corollary \ref{CorLDendViaRB}.
\end{proof}


\end{document}